\documentclass[11pt]{article}
\usepackage[inner=3.4 cm,outer=3.4 cm,top=3.0 cm,bottom=2.0 cm]{geometry}

\usepackage{amsmath,amssymb,amsfonts,amsthm}
\usepackage{setspace}
\newenvironment{myabstract}{\par\noindent
{\bf Abstract . } \small }
{\par\vskip8pt minus3pt\rm}
\newcounter{item}[section]
\newcounter{kirshr}
\newcounter{kirsha}
\newcounter{kirshb}
\newenvironment{enumarab}{\setcounter{kirshb}{1}
\begin{list}{(\arabic{kirshb})}{\usecounter{kirshb}} }{\end{list}}
\newtheorem{theorem}{Theorem}[section]

\newtheorem{lemma}[theorem]{Lemma}
\newtheorem{corollary}[theorem]{Corollary}
\newenvironment{demo}[1]{\noindent{\bf #1.}\upshape\mdseries}
{\nopagebreak{\hfill\rule{2mm}{2mm}\nopagebreak}\par\normalfont}
\theoremstyle{definition}

\newtheorem{example}[theorem]{Example}
\newtheorem{definition}[theorem]{Definition}
\def\R{\mathbb{R}}

\def\C{{\mathfrak{C}}}
\def\Fm{{\mathfrak{Fm}}}

\def\At{{\bf At}}
\def\Nr{{\mathfrak{Nr}}}

\def\Sg{{\mathfrak{Sg}}}
\def\Fm{{\mathfrak{Fm}}}
\def\A{{\mathfrak{A}}}
\def\B{{\mathfrak{B}}}
\def\C{{\mathfrak{C}}}
\def\D{{\mathfrak{D}}}
\def\M{{\mathfrak{M}}}
\def\N{{\mathfrak{N}}}
\def\Sn{{\mathfrak{Sn}}}
\def\CA{{\bf CA}}
\def\QA{{\bf QA}}

\def\QEA{{\bf QEA}}

\def\Df{{\bf Df}}

\def\Lf{{\bf Lf}}

\def\PA{{\bf PA}}
\def\PEA{{\bf PEA}}

\def\K{{\bf K}}
\def\K{{\bf K}}

\def\RCA{{\bf RCA}}

\def\Rd{{\ Rd}}
\def\(R)RA{{\bf (R)RA}}
\def\RA{{\bf RA}}

\def\R{\mathbb{R}}

\def\Sc{{\bf Sc}}

\def\Id{{\bf Id}}
\def\c #1{{\cal #1}}
 \def\CA{{\sf CA}}
\def\B{{\sf B}}

\def\g{{\sf g}}
\def\b{{\sf b}}
\def\r{{\sf r}}
\def\K{{\sf K}}
\def\tp{{\sf tp}}

 \def\Cm{{\mathfrak{Cm}}}
\def\Nr{{\mathfrak{Nr}}}

\def\restr #1{{\restriction_{#1}}}
\def\cyl#1{{\sf c}_{#1}}
\def\diag#1#2{{\sf d}_{#1#2}}
\def\sub#1#2{{\sf s}^{#1}_{#2}}
\def\R{\sf R}
\def\Ra{{\mathfrak{Ra}}}
\def\Ca{{\mathfrak{Ca}}}
\def\set#1{\{#1\} }
\def\Ra{{\mathfrak{Ra}}}
\def\Nr{{\mathfrak{Nr}}}
\def\Tm{{\mathfrak{Tm}}}
\def\A{{\mathfrak{A}}}
\def\B{{\mathfrak{B}}}
\def\C{{\mathfrak{C}}}
\def\D{{\mathfrak{D}}}

\def\A{{\mathfrak{A}}}
\def\B{{\mathfrak{B}}}
\def\C{{\mathfrak{C}}}
\def\D{{\mathfrak{D}}}

\def\U{{\mathfrak{U}}}

\def\Bb{{\mathfrak{Bb}}}
\def\L{{\mathfrak{L}}}
\def\Rd{{\mathfrak{Rd}}}
\def\Bb{{\mathfrak{Bb}}}
\def\At{{\mathfrak{At}}}
\def\L{{\mathfrak{L}}}

\def\CA{{\bf CA}}
\def\RA{{\bf RA}}

\def\RCA{{\bf RCA}}

\def\F{{\mathfrak{F}}}
\def\At{{\sf{At}}}
\def\N{\mathbb{N}}
\def\R{\mathfrak{R}}

\def\Cs{{\sf Cs}}

\def\sub#1#2{{\sf s}^{#1}_{#2}}
\def\cyl#1{{\sf c}_{#1}}
\def\diag#1#2{{\sf d}_{#1#2}}

\def\c #1{{\cal #1}}

\def\pa{$\forall$}
\def\pe{$\exists$}

\def\ef{Ehren\-feucht--Fra\"\i ss\'e}

\def\nodes{{\sf nodes}}

\def\restr #1{{\restriction_{#1}}}

\def\Ra{{\mathfrak{Ra}}}
\def\Nr{{\mathfrak{Nr}}}
\def\Z{{\cal Z}}
\def\CA{{\bf CA}}
\def\RCA{{\bf RCA}}
\def\c#1{{\mathcal #1}}

\def\set#1{ \{#1\}}

\def\Ca{{\mathfrak Ca}}
\def\b#1{{\bar{ #1}}}
\def\pe{$\exists$}
\def\pa{$\forall$}
\def\Cm{{\mathfrak Cm}}
\def\Sg{{\mathfrak Sg}}

\def\Rl{{\mathfrak Rl}}
\def\N{{\cal N}}

\def\ls { L\"owenheim--Skolem}
\def\At{{\sf At}}
\def\Uf{{\sf Uf}}

\def\rng{{\sf rng}}
\def\dom{{\sf dom}}
\def\Cm{{\sf Cm}}

\def\g{{\sf g}}

\def\r{{\sf r}}
\def\tp{{\sf tp}}
\def\cyl#1{{\sf c}_{#1}}
\def\sub#1#2{{\sf s}^{#1}_{#2}}
\def\diag#1#2{{\sf d}_{#1#2}}

\def\ws{winning strategy}
\def\ef{Ehren\-feucht--Fra\"\i ss\'e}

 \def\CA{{\sf CA}}
\def\Cs{{\sf Cs}}
\def\RCA{{\sf RCA}}

\def\RA{{\sf RA}}
\def\PA{{\sf PA}}

\def\PEA{\sf PEA}
\def\QEA{{\sf QEA}}

\def\g{{\sf g}}
\def\r{{\sf r}}

\def\Z{{\mathbb{Z}}}
\def\N{{\mathbb{N}}}

\def\CPEA{\sf CPEA}
\def\U{{\mathfrak{U}}}

\def\Gp{{\sf Gp}}

\def\c{{\sf c}}
\def\s{{\sf s}}

\def\Id{{\sf Id}}
\def\Sc{{\sf Sc}}
\def\Df{{\sf Df}}
\def\Lf{{\sf Lf}}
\def\K{{\sf K}}
\def\nodes{{\sf nodes}}

\def\Sc{{\sf Sc}}
\def\Df{{\sf Df}}
\def\PA{{\sf PA}}
\def\Id{{\sf Id}}
\def\QEA{{\sf QEA}}
\def\s{{\sf s}}

\def\CA{{\sf CA}}
\def\K{{\sf K}}
\def\QA{{\sf QA}}
\def\RCA{{\sf RCA}}

\def\A{{\mathfrak{A}}}
\def\Cs{{\sf Cs}}

\def\V{{\sf V}}

\def\cyl#1{{\sf c}_{#1}}
\def\sub#1#2{{\sf s}^{#1}_{#2}}
\def\diag#1#2{{\sf d}_{#1#2}}
\def\Mo{{\sf Mo}}
\def\swap#1#2{{\sf s}_{[#1, #2]}}

\def\la#1{\langle#1\rangle}

\def\Nr{{\sf Nr}}
\def\de{Dedekind-MacNeille}
\def\Cm{{\mathfrak{Cm}}}
\def\M{{\sf M}}
\def\T{{\sf T}}

\def\CRCA{{\sf CRCA}}
\def\PEA{{\sf PEA}}
\def\RPA{{\sf RPA}}
\def\Nrr{{\mathfrak{Nr}}}
\title{Completely representable neat reducts 
}
\author{Tarek Sayed Ahmed\\
Department of Mathematics, Faculty of Science,\\
Cairo University, Giza, Egypt.
 }
\date{}
\begin{document}
\maketitle

\begin{myabstract} For an ordinal $\alpha$, $\sf PEA_{\alpha}$ denotes the class of polyadic equality algebras of dimension $\alpha$. 
We show that for several classes of algebras that are reducts of $\PEA_{\omega}$ whose signature contains all substitutions and finite cylindrifiers, if $\B$ is  in such a class, and $\B$ is atomic, 
then for all $n<\omega$,  $\Nr_n\B$ is completely representable as a $\PEA_n$. Conversely, we show that for any $2<n<\omega$, and any variety $\sf V$, 
between diagonal free cylindric algebras and quasipolyadic equality algebras of dimension $n$, 
the class of completely representable algebras in $\sf V$ is not elementary.
\end{myabstract}

\section{Introduction} Relation algebras $\sf RA$s and cylindric algebras of dimension $\alpha$, $\alpha$ any ordinal $\CA_{\alpha}$ are introduced by Tarski. Both are varieties
that are axiomatized by a relatively simple schema of equations. 
Relation algebras are abstractions of algebras whose universe consists of binary relations, with top element an equivalence relation, and Boolean operations of union and complementation  
and extra operations of composition and forming converses. Such algebras are called representable relation algebras of dimension $\alpha$, in symbols $\sf RRA$. 
In both cases equality is represented by the identity relation. The last class, when the top elements are disjoint unions of cartesian squares of dimension $\alpha$ 
is called the class of representable polyadic algebras  algebras of dimension $\alpha$, and is denoted by $\sf RCA_{\alpha}$.
Unless otherwise explicitly indicated, let $2<n<\omega$. The classes $\sf RRA$ and $\RCA_n$ are not finitely axiomatizable,  
In particular $\sf RRA\subsetneq \sf RA$ and similarly $\RCA_n\subsetneq \CA_n$. 
Polyadic algebras were introduced by Halmos \cite{Halmos} to provide an algebraic reflection
of the study of first order logic without equality. Later the algebras were enriched by 
diagonal elements to permit the discussion of equality. 
That the notion is indeed an adequate reflection of first order logic was 
demonstrated by Halmos' representation theorem for locally finite polyadic algebras 
(with and without equality). Daigneault and Monk  
proved a strong extension of Halmos' theorem, namely that,  
every polyadic algebra of infinite dimension (without equality) is representable \cite{DM}. 

In the realm of representable algebras, there are several types of representations. Ordinary representations are just isomorphisms from Boolean algebras
with operators to a more concrete structure (having the same signature) 
whose elements are sets endowed with set-theoretic operations like intersection and complementation.
Complete representations, on the other hand, are representations that preserve arbitrary conjunctions whenever defined.
More generally consider the following question: Given an  algebra  and a set of meets, is there a representation  that carries 
this set of meets to set theoretic intersections? A complete representation would thus be one that preseves {\it all existing} meets (finite of course and infinite). 
Here we are assuming that our semantics is specified by set algebras, with the concrete Boolean operation of intersection among
its basic operations. 
When the algebra in question is countable, and we have only countably many meets; 
this is an algebraic version of an  omitting types theorem; the representation omits the given set meets or non-principal types. 
When the algebra in question is atomic, then a representation omitting the non-principal type consisting of co-atoms,  turns out to be a complete representation.
This follows from the following result due to Hirsch and Hodkinson: A Boolean algebra $\A$ has a complete representation $f:\A\to \langle\wp(X), \cup, \cap, \sim, \emptyset, X\rangle$ 
($f$ is a 1-1 homomorphism and $X$ a set) $\iff$  
$\A$ atomic and $\bigcup_{x\in \At\A} f(x)=X$, where $\At\A$ is the set of atoms of $\A$. 

On the face of it, the notion of complete representations seems to be strikingly a second order one. This intuition is confirmed in \cite{HH} where it is proved
that the classes of completely representable cylindric algebras of dimension at least three and that of relation algebras are  not elementary. 
These results were proved by Hirsch and Hodkinson using so-called rainbow algebras \cite{HH}; in this paper we present  entirely
different proofs for all such results and some more closely related ones using so called Monk-like algebras. Our proof depends 
essentially on some form of an infinite combinatorial version of Ramsey's Theorem. 
But running to such conclusions--concerning (non-)first order definablity-- can be reckless and far too hasty; 
for in other non-trivial cases the notion of complete representations turns {\it not to be} a genuinely second order one; it is 
definable in first order logic. 
The class of completely representable
Boolean algebras is elementary; it simply coincides with the atomic ones.  
A far less trivial example is the class of completely representable infinite dimensional polyadic algebras; 
it coincides with the class of atomic, completely additive algebras. It is not hard to show that, like atomicity, complete additivity can indeed be defined in first order logic \cite{au}.
This is not true for the class $\PEA_{\alpha}$ of {\it polyadic algebras with equality of dimension $\alpha$.} 
However, we will show that if $\A\in \PEA_{\alpha}$ is atomic, then {\it all of its  finite dimensional neat reducts are not only representable, but completely 
representable.} So from one atomic algebra one obtains 
a plethora of completely representable ones,  at least one for each finite dimension .

For some odd reason, historically the underlying intuition of the notion of complete representability 
progressed  in a different direction. The correlation of (the first order property of) atomicity to complete representations has caused a lot of confusion in the past.
It was mistakenly thought for a while, among algebraic logicians,  that atomic representable relation and cylindric algebras
are completely representable, an error attributed to Lyndon and now referred to as Lyndon's error. But in retrospect, one can safely say by gathering and scrutinizing recent results that 
the first order definability of the the notion of complete representations heavily depends on the algebras required to be completely represented. 
In other words, the (possibly slippery) 
notion of `complete representations'  
needs a context to be fixed one way or another, and it is surely unwise to declare a verdict without a careful and thorough investigation of 
the specific situation at hand. Let $\sf CRRA$ denote the class of completely representable $\sf RA$S, and $\CRCA_n$ denote the class of completely reprsentable $\CA_n$s. 
It is shown that the classes   $\sf CRRA$ and $\CRCA_n$ are not elmentary, reproving a result of Hirsch and Hodkinson, 
and we fo further by showing that $\sf CRRA$  is not even closed 
under $\equiv_{\infty, \omega}$.

In \cite{IGPL} it is proved that for any pair of infinite ordinals $\alpha<\beta$, the class $\Nr_{\alpha}\CA_{\beta}$ is not elementary. 
A different model theoretic proof for finite $\alpha$ is given in \cite[Theorem 5.4.1]{Sayedneat}.
This result is extended to many cylindric like algebras like Halmos'  polyadic algbras with and without eqaulity, and Pinter's substitution algebras in \cite{Fm, note}. The class $\CRCA_n$ is 
proved not be elementary by Hirsch and Hodkinson in \cite{HH}. Neat embeddings and complete representations are linked in \cite[Theorem 5.3.6]{Sayed} where it is shown that 
$\CRCA_n$ coincides with the class $\bold S_c\Nr_n\CA_{\omega}$ on atomic 
algebra having countably many atoms. In \cite{jsl} it is proved that this charactarization does not generalize to atomic algebras having uncountably many atoms.  
 Such counterexamples are used to violate metalogical theorems such as \cite[Theorem 3.2.9-10 ]{Sayed} involving the celebrated Orey-Henkn omitting types theorem for finite varible fragment 
of $L_{\omega,\omega}$.

\section{Preliminaries}

We follow the notation of \cite{1} which is in conformity with the notation in the monograph 
\cite{HMT2}.  In particular, for any pair of ordinal $\alpha<\beta$, $\CA_{\alpha}$ stands for the class of cylindric algebras of dimension
$\alpha$, $\RCA_{\alpha}$ denotes the class of representable $\CA_{\alpha}$s and $\Nr_{\alpha}\CA_{\beta}(\subseteq \CA_{\alpha})$ 
denotes the class of $\alpha$--neat reducts of $\CA_{\beta}$s. The last class is studied extensively in 
the chapter  \cite{Sayedneat} of \cite{1} as a key notion in the representation theory of cylindric algebras. 
The notion of {\it neat reducts} and the related one of {\it neat embeddings} are both important in algebraic logic for the
simple reason that both notions are very much tied
to the notion of representability, via the so--called neat embedding theorem of Henkin's which says that (for any ordinal 
$\alpha$),  we have $\sf RCA_{\alpha}=\bold S\Nr_{\alpha}\CA_{\alpha+\omega}$, where $\bold S$ stands for the operation of forming subalgebras.
\begin{definition} 
Assume that $\alpha<\beta$ are ordinals and that 
$\B\in \CA_{\beta}$. Then the {\it $\alpha$--neat reduct} of $\B$, in symbols
$\mathfrak{Nr}_{\alpha}\B$, is the
algebra obtained from $\B$, by discarding
cylindrifiers and diagonal elements whose indices are in $\beta\setminus \alpha$, and restricting the universe to
the set $Nr_{\alpha}B=\{x\in \B: \{i\in \beta: {\sf c}_ix\neq x\}\subseteq \alpha\}.$
\end{definition}

It is straightforward to check that $\mathfrak{Nr}_{\alpha}\B\in \CA_{\alpha}$. 
Let $\alpha<\beta$ be ordinals. If $\A\in \CA_\alpha$ and $\A\subseteq \mathfrak{Nr}_\alpha\B$, with $\B\in \CA_\beta$, then we say that $\A$ {\it neatly embeds} in $\B$, and 
that $\B$ is a {\it $\beta$--dilation of $\A$}, or simply a {\it dilation} of $\A$ if $\beta$ is clear 
from context. For $\bold K\subseteq \CA_{\beta}$, we write $\Nr_{\alpha}\bold K$ for the class $\{\mathfrak{Nr}_{\alpha}\B: \B\in \bold K\}.$
  
Fix $2<n<\omega$.  
Following \cite{HMT2}, ${\sf Cs}_n$ denotes the class of cylindric set algebras of dimension $n$, and ${\sf Gs}_n$ 
denotes the class of generalized set algebra of dimension $n$; $\C\in {\sf Gs}_n$, if $\C$ has top element
$V$ a disjoint union of cartesian squares,  that is $V=\bigcup_{i\in I}{}^nU_i$, $I$ is a non-empty indexing set, $U_i\neq \emptyset$  
and  $U_i\cap U_j=\emptyset$  for all $i\neq j$. The operations of $\C$ are defined like in cylindric set algebras of dimension $n$ 
relativized to $V$. 
${\sf CRCA_n}$ denotes the class of completely represenatble $\CA_n$s.
\begin{definition} An algebra $\A\in {\sf CRCA}_n$ $\iff$ there exists $\C\in {\sf Gs}_n$, and an isomorphism $f:\A\to \C$ such that for all $X\subseteq \A$, 
$f(\sum X)=\bigcup_{x\in X}f(x)$, whenever $\sum X$ exists in $\A$. In this case, we say that $\A$ is completely representable via $f$.
\end{definition}
It is known tht $\A$ is completely representable via $f:\A\to \C$, where $\C\in {\sf Gs}_n$ has top element $V$ say 
$\iff$ $\A$ is atomic and $f$ is {\it atomic} in the sense that 
$f(\sum \At\A)=\bigcup_{x\in \At\A}f(x)=V$ \cite{HH}.
  $\bold S_c$ denotes the operation of forming {\it complete} subalgebras. 
The next lemma tells us that the notions of atomicity and complete representation of an algebra 
are inherited by complete (hence dense) sublgebras. 

\begin{lemma}\label{join} Let $n<\omega$, and $\D$ be a Boolean algebra.
Assume that $\A\subseteq_c \D$. If $\D$ is atomic, then $\A$ is atomic
\cite[Lemma 2.16]{HHbook}.  If $\D\in \CA_n$ is completely representable, then so is $\A$.
\end{lemma}
\begin{proof}
Let everything be as in the hypothesis of the first part. We show that $\A$ is atomic. Let $a\in A$ be non--zero. Then since $\D$ is atomic, 
there exists an atom $d\in D$, such that $d\leq a$. 
Let $F=\{x\in A: x\geq d\}$. Then $F$ is an ultrafilter of $\A$. 
It is clear that $F$ is a filter. To prove maximality, assume that $c\in A$ and $c\notin F$, then $-c\cdot d\neq 0$, so $0\neq -c\cdot d\leq d$, hence $-c\cdot d=d$, 
because $d$ is an atom in $\B$, thus $d\leq -c$, and we get by definition that $-c\in F$.  We have shown that $F$ is an ultrafilter. 
We now show  that $F$ is a principal ultrafilter in $\A$, 
that is, it is generated by an atom. Assume for contradiction that it is not, so that $\prod^{\A} F$ exists, because $F$ is an ultrafilter 
and $\prod ^{\A}F=0$, because it is non--principal. 
But $\A\subseteq_c \D$, so we obtain $\prod^{\A}F=\prod^{\D}F=0$. This contradicts that $0<d\leq x$ for all $x\in F$.
Thus $\prod^{\A}F=a'$, $a'$ is an atom in $\A$, $a'\in F$ 
and $a'\leq a$, because $a\in F$.  We have proved the first required.
Let  $\A\subseteq_c \D$ and 
assume that $\D$ is completely representable.  We will show that $\A$ is completely representable. 
Let $f:\D\to \wp(V)$ be a complete representation of
$\D$. 
We claim that $g=f\upharpoonright \A$ is a complete representation of $\A$. 
Let $X\subseteq \A$ be such that $\sum^{\A}X=1$. 
Then by $\A\subseteq_c \D$, we have  $\sum ^{\D}X=1$. Furthermore, for all $x\in X(\subseteq \A)$ we have $f(x)=g(x)$, so that 
$\bigcup_{x\in X}g(x)=\bigcup_{x\in X} f(x)=V$, since $f$ is a complete representation, 
and we are done. 
\end{proof}
Though the class $\bold S_c\Nr_n\CA_{\omega}$ and the class $\sf CRRA$ coincide on algebras having countably many atoms,
in \cite{bsl} it is shown that  
the condition of countability cannot be omitted: There is an atomic $\A\in {\sf Nr}_n\CA_{\omega}$ with uncountably many atoms such that $\A$ is not completely representable 
But the $\C\in \CA_{\omega}$ for which $\A=\Nr_n\C$ is atomless. 

In what follows we adress complete representability of a given algebra 
in connection to the existence of an $\omega$--dilation of this algebra  that is atomic. We shall deal with many classes of cylindric--like algebras for which the neat reduct 
operator can be defined. In particular, for such classes, and regardless of atomicity, we can (and will) 
talk about an $\omega$--dilation of a given algebra.
For an ordinal $\alpha$, let $\PA_{\alpha}(\PEA_{\alpha})$ denote the class of $\alpha$--dimensional polyadic (equality) algebas as defined in \cite[Definition 5.4.1]{HMT2}.

\begin{definition} Let $\alpha$ be an ordinal. By a {\it polyadic algebra} of dimension $\alpha$,  
or a $\PA_{\alpha}$ for short,
we understand an algebra of the form
$$\A=\langle A,+,\cdot ,-,0,1,{\sf c}_{(\Gamma)},{\sf s}_{\tau} \rangle_{\Gamma\subseteq \alpha ,\tau\in {}^{\alpha}\alpha}$$
where ${\sf c}_{(\Gamma)}$ ($\Gamma\subseteq  \alpha$) and ${\sf s}_{\tau}$ ($\tau\in {}^{\alpha}\alpha)$ are unary 
operations on $A$, such that postulates  
below hold for $x,y\in A$, $\tau,\sigma\in {}^{\alpha}\alpha$ and 
$\Gamma, \Delta\subseteq \alpha$ 

\begin{enumerate}
\item $\langle A,+,\cdot, -, 0, 1\rangle$ is a boolean algebra
\item ${\sf c}_{{(\Gamma)}}0=0$

\item $x\leq {\sf c}_{{(\Gamma)}}x$

\item ${\sf c}_{(\Gamma)}(x\cdot {\sf c}_{(\Gamma)}y)={\sf c}_{(\Gamma)}x \cdot {\sf c}_{(\Gamma)}y$

\item ${\sf c}_{(\Gamma)}{\sf c}_{(\Delta)}x={\sf c}_{(\Gamma\cup \Delta)}x$

\item ${\sf s}_{\tau}$ is a boolean endomorphism

\item ${\sf s}_{Id}x=x$

\item ${\sf s}_{\sigma\circ \tau}={\sf s}_{\sigma}\circ {\sf s}_{\tau}$

\item if $\sigma\upharpoonright (\alpha\sim \Gamma)=\tau\upharpoonright (\alpha\sim \Gamma)$, then 
${\sf s}_{\sigma}{\sf c}_{(\Gamma)}x={\sf s}_{\tau}{\sf c}_{(\Gamma)}x$


\item If $\tau^{-1}\Gamma=\Delta$ and $\tau\upharpoonright \Delta $ is
one to one, then ${\sf c}_{(\Gamma)}{\sf s}_{\tau}x={\sf s}_{\tau}{\sf c}_{(\Delta )}x$.

\end{enumerate}

\end{definition}

\begin{definition}Let $\alpha$ be an ordinal. By a {\it polyadic equality algebra} of dimension $\alpha$,  
or a $\PEA_{\alpha}$ for short,
we understand an algebra of the form
$$\A=\langle A,+,\cdot ,-,0,1,{\sf c}_{(\Gamma)},{\sf s}_{\tau}, {\sf d}_{ij} \rangle_{\Gamma\subseteq \alpha ,\tau\in {}^{\alpha}\alpha, i,j<\alpha}$$
where ${\sf c}_{(\Gamma)}$ ($\Gamma\subseteq  \alpha$) and ${\sf s}_{\tau}$ ($\tau\in {}^{\alpha}\alpha)$ are unary 
operations on $A$, and ${\sf d}_{ij}$ are constants in the signature, such the reduct obtained by deleting these ${\sf d}_{ij}$'s ($i, j\in \alpha$), 
is a $\PA_{\alpha}$ and  
the equations   
below hold for $x\in A$, $\tau\in {}^{\alpha}\alpha$ and 
and $i, j\in \alpha$ 
\begin{enumerate}

\item ${\sf d}_{ii}=1$,

\item $x\cdot {\sf d}_{ij}\leq {\sf s}_{[i|j]}x$,

\item ${\sf s}_{\tau}{\sf d}_{ij}={\sf d}_{\tau(i), \tau(j)}$.

\end{enumerate}
\end{definition}

We will sometimes add superscripts to cylindrifiers and substitutions indicating the algebra they are evaluated in.
The class of representable algebras is defined via set - theoretic operations
on sets of $\alpha$-ary sequences. Let $U$ be a set. 
For $\Gamma\subseteq \alpha$ and $\tau\in {}^{\alpha}\alpha$, we set
$${\sf c}_{(\Gamma)}X=\{s\in {}^{\alpha}U: \exists t\in X,\ \forall j\notin \Gamma, t(j)=s(j) \}$$
and 
$${\sf s}_{\tau}X=\{s\in {}^{\alpha}U: s\circ \tau\in X\}.$$
$${\sf D}_{ij}=\{s\in {}^{\omega}U: s_i=s_j\}.$$
For a set $X$, let $\B(X)$ be the boolean set algebra $(\wp(X), \cup, \cap, \sim).$
The class of representable polyadic algebras, or ${\RPA}_{\alpha}$ for short, is defined by
$$SP\{\langle \B(^{\alpha}U), {\sf c}_{(\Gamma)}, {\sf s}_{\tau} \rangle_{\Gamma\subseteq \alpha, \tau\in {}^{\alpha}\alpha}:
\ \,  U \text { a set }\}.$$
The class of representable polyadic equality algebras, or ${\sf RPEA}_{\alpha}$ for short, is defined by
$$SP\{\langle \B(^{\alpha}U), {\sf c}_{(\Gamma)}, {\sf s}_{\tau}, {\sf D}_{ij} \rangle_{\Gamma\subseteq \alpha, \tau\in {}^{\alpha}\alpha}:
\ \,  U \text { a set }\}.$$

Here $SP$ denotes the operation of forming subdirect products. It is straightforward to show that ${\RPA}_{\alpha}\subseteq \PA_{\alpha}.$ 
Daigneault and Monk \cite{DM} proved that for $\alpha\geq \omega$ the converse inclusion also holds, 
that is ${\RPA}_{\alpha}=\PA_{\alpha}.$ 
This is a completeness theorem for certain infinitary extensions of first order logic 
without equality \cite{K}.
Let $\A$ be a polyadic algebra and $f:\A\to \wp(^{\alpha}U)$ be a representation of $\A$.
If $s\in X$, we let
$$f^{-1}(s)=\{a\in \A: s\in f(a)\}.$$
An atomic representation $f:\A\to \wp(^{\alpha}U)$ is a representation such that for each 
$s\in V$, the ultrafilter $f^{-1}(s)$ is principal. 
A complete representation of $\A$ is a representation $f$ satisfying
$$f(\prod X)=\bigcap f[X]$$
whenever $X\subseteq \A$ and $\prod X$ is defined.

A completely additive boolean algebra with operators is one for which all extra non-boolean operations preserve arbitrary joins.

\begin{lemma}\label{r} Let $\A\in \PA_{\alpha}$. A representation $f$ of $\A$
is atomic if and only if it is complete. If $\A$ has a complete representation, then it is atomic and is completely additive.
\end{lemma}
\begin{demo}{Proof} The first part is like \cite{HH}.
For the second part, we note that $\PA_{\alpha}$ is a discriminator variety with discriminator term ${\sf c}_{(\alpha)}$.
And so because all algebras in $\PA_{\alpha}$ are semi-simple,
it suffices to show that if $\A$ is simple, $X\subseteq A$, is such that $\sum X=1$,
and there exists  an injection $f:\A\to \wp(^{\alpha}\alpha)$, such that $\bigcup_{x\in X}f(x)=V$,
then for any $\tau\in {}^{\alpha}\alpha$, we have $\sum s_{\tau}X=1$. So assume that this does not happen 
for some $\tau \in {}^{\alpha}\alpha$.
Then there is a $y\in \A$, $y<1$, and
$s_{\tau}x\leq y$ for all $x\in X$.
Now
$$1=s_{\tau}(\bigcup_{x\in X} f(x))=\bigcup_{x\in X} s_{\tau}f(x)=\bigcup_{x\in X} f(s_{\tau}x).$$
(Here we are using that $s_{\tau}$ distributes over union.)
Let $z\in X$, then $s_{\tau}z\leq y<1$, and so $f(s_{\tau}z)\leq f(y)<1$, since $f$ is injective, it cannot be the case that $f(y)=1$.
Hence, we have
$$1=\bigcup_{x\in X} f(s_{\tau}x)\leq f(y) <1$$
which is a contradiction, and we are done.

\end{demo}

Let $\alpha$ be an infinite ordinal. By \cite{au}, it is proved that  the the condition of atomicity and complete
additivity of a $\PA_{\alpha}$ is not only necessary for completely representability but also sufficient.
But for $\PEA_{\alpha}$ the situation is totally diferent. Not only not every atomic $\PEA_{\alpha}$ is not completely representable; 
there are examples of atomic $\PEA_{\alpha}$s that are not reprsentable at all. The aim of this paper is that if we pass to finte neat reducts 
of any atomic $\PEA_{\alpha}$ possibly non representable,  
we recover complete representability.

\subsection{Classes between $\Df_n$ and $\QEA_n$}

We shall have the occasion to deal with (in addition to $\CA$s), the following cylindric--like algebras \cite{1}: $\sf Df$ short for diagonal free cylindric algebras, $\sf Sc$ short for Pinter's substitution algebras,  
$\sf QA$($\QEA$) short for quasi--polyadic (equality) algebras.  
For $\K$ any of these classes and $\alpha$ any ordinal, 
we write $\K_{\alpha}$ for variety of $\alpha$--dimensional $\K$ algebras which can be axiomatized by a finite schema of equations, 
and $\sf RK_{\alpha}$ for the class of  representable $\K_{\alpha}$s, which happens to be a variety too (that cannot be axiomatized by a finite schema of equations for $\alpha>2$ unless $\sf K=\sf PA$ and $\alpha\geq \omega$).
The standard reference for all the classes of algebras mentioned previously  is  \cite{HMT2}. 
We recall the concrete verions of such algebras.
Let $\tau:\alpha\to \alpha$
and $X\subseteq {}^{\alpha}U,$  then 
$${\sf S}_{\tau}X=\{s\in {}^{\alpha}U: s\circ \tau\in X\}.$$
For $i,j\in \alpha$, $[i|j]$ is the replacement on $\alpha$ that sends $i$ to $j$ and is the identity map on $\alpha\sim \{i\}$ while $[i,j]$ is the transposition on $\alpha$ that interchanges $i$ and $j$. 
\begin{itemize}
\item A {\it diagonal free cylindric set algebra of dimension $\alpha$} is an algebra of the form 
$\langle \B(^{\alpha}U),  {\sf C}_i\rangle_{i,j<\alpha}.$ 
\item A {\it Pinter's substitution et algebra of dimension $\alpha$} is an algebra of the form 
$\langle \B(^{\alpha}U),  {\sf C}_i,  {\sf S}_{[i|j]}\rangle_{i,j<\alpha}.$ 
\item A {\it quasi-polyadic set algebra of dimension $\alpha$} is an algebra of the form\\ 
$\langle \B(^{\alpha}U),  {\sf C}_i,  {\sf S}_{[i|j]}, {\sf S}_{[i,j]}\rangle_{i,j<\alpha}.$ 
\item A {\it quasi-polyadic equality set algebra} is an algebra of the form\\ 
$\langle \B(^{\alpha}U),  {\sf C}_i,  {\sf S}_{[i|j]}, {\sf S}_{[i,j]}, {\sf D}_{ij}\rangle_{i,j<\alpha}$.
\end{itemize}

\begin{figure}
\[\begin{array}{l|l}
\mbox{class}&\mbox{extra non-Boolean operators}\\
\hline
\Df_{\alpha}& \cyl i: i<\alpha\\
\Sc_\alpha&\cyl i, \s_i^j :i, j<\alpha\\
\CA_\alpha&\cyl i, \diag i j: i, j<\alpha\\
\PA_\alpha&\cyl i, \s_\tau: i<n,\; \tau\in\;^\alpha\alpha\\
\PEA_\alpha&\cyl i, \diag i j,  \s_\tau: i, j<n,\;  \tau\in\;^\alpha\alpha\\
\QA_\alpha&  \cyl i, \s_i^j, \s_{[i, j]} :i, j<\alpha  \\
\QEA_\alpha&\cyl i, \diag i j, \s_i^j, \s_{[i, j]}: i, j<\alpha
\end{array}\]
\caption{Non-Boolean operators for the classes\label{fig:classes}}
\end{figure}

For a $\sf BAO$, $\A$ say, for any ordinal $\alpha$, $\Rd_{ca}\A$ denotes the cylindric reduct of $\A$ if it has one, $\Rd_{sc}\A$
denotes the $\Sc$ reduct of $\A$ if it has one, and
$\Rd_{df}\A$ denotes the reduct of $\A$ obtained by discarding all the operations except for cylindrifications.
If $\A$ is any of the above classes, it is always the case that $\Rd_{df}\A\in \sf Df_{\alpha}$. If $\A\in \CA_\alpha$, then $\Rd_{sc}\A\in \Sc_\alpha$, and if $\A\in \QEA_{\alpha}$ then $\Rd_{ca}\A\in \CA_{\alpha}$.
Roughly speaking for an ordinal $\alpha$, $\CA_{\alpha}$s are  not expansions of $\Sc_{\alpha}$s, but they are {\it definitionally equivalent} to expansions of $\Sc_{\alpha}$, 
because the $\s_i^j$s are term definable in $\CA_{\alpha}$s by 
$\s_i^j(x)=\c_i(x\cdot {-\sf d}_{ij})$ $(i,j<\alpha)$. This operation reflects algebraically the subsititution of the variable $v_j$ for $v_i$ in a formula such that the substitution is free; this can be always done by reindexing bounded variables.
In such  situation, we say that $\Sc$s are {\it generalized reducts} of $\CA$s. However, $\CA_{\alpha}$s and $\sf \QA_{\alpha}$ are (real )reducts of $\QEA$s, (in the universal algebraic sense) simply obtained by discarding 
the operations in their signature not in the signature of their common expansion 
$\QEA_{\alpha}$.
We give a finite  approximate equational axiomatization of the concrete algebras defined above, which are the prime source of inspiration for these axiomatizations introduced to capture representability.
However, like for $\CA$s, this works only for certain special cases like the locally finite algebras, but does not generalize much further, cf Proposition \ref{j}.
\begin{definition}\label{def:qpea}
\begin{description}
\item[Substitution Algebra, $\Sc$]
\cite{Pinter}.

Let $\alpha$ be an ordinal. By a substitution algebra of dimension $\alpha$, briefly an $\Sc_{\alpha},$ we mean an algebra
$${\A}=\langle  A,+, \cdot, - , 0, 1, {\sf c}_i, \sub i j \rangle_{i, j<\alpha}$$
where $\langle A, +, \cdot, -, 0, 1\rangle$ is a Boolean algebra, $\cyl i, \sub i j$ are unary operations on $\A$ (for $i, j<\alpha$) satisfying
the following equations for all $i,j,k,l < \alpha$:

\begin{enumerate}

\item \label{en:s1}$\cyl i 0=0,\; x\leq \cyl i x, \; \cyl i(x\cdot \cyl i y)
=\cyl i x\cdot \cyl i y$, and $\cyl i\cyl j x =\cyl j\cyl i x,$

\item $\sub i i x=x,$\label{en:id}

\item $\sub i j$ is a boolean endomorphisms,\label{en:end}

\item   $\sub i j\cyl ix=\cyl ix,$\label{en:s3}

\item  $\cyl i\sub i jx=\sub i jx$ whenever $i\neq j,$

\item  $\sub i j\cyl kx=\cyl k\sub i j x$, whenever $k\notin \{i,j\},$

\item  $\cyl i\sub j i x=\cyl j\sub i j x,$

\item $\sub j i \sub l k x=\sub l k \sub j ix$, whenever $|\{i,j,k,l\}|=4,$

\item  $\sub l i \sub j l x=\sub l i \sub j i x$.\label{en:s9}
\end{enumerate}

\item[Quasipolyadic algebra, $\QEA$]  \cite{ST}.  

A quasipolyadic algebra of dimension $ \alpha$, briefly a
$\QA_{\alpha}$, is an algebra $$\A=\langle A,+, \cdot, -, 0, 1, \cyl i,
\sub i j, \swap i j\rangle_{i, j<\alpha}$$
where the reduct to $\Sc_\alpha$ is a substitution algebra (it satisfies \eqref{en:s1}--\eqref{en:s9} above) and additionally it satisfies  the following
equations for all $i, j, k <  \alpha$:
\begin{enumerate}
\item [\ref{en:id}'] $\sub i i (x)=\swap i i (x)=x,\text { and }\swap i j=\swap j i,$
\item [\ref{en:end}'] $\sub i j$ and $\swap i j$ are boolean endomorphisms,
\item  $\swap i j\swap i j x=x$,
\item $\swap i j\swap i k=\swap jk\swap ij  ~~ \text { if
}~~ |\{i,j,k\}|=3$,
\item  $\swap i j \sub j i x=\sub i j x$.\label{en:12}
\end{enumerate}

\item[Quasipolyadic equality algebra, $\QEA$]\cite{ST}.

A quasipolyadic equality algebra of dimension $\alpha$, briefly
a $\QEA_{ \alpha}$ is an algebra $$\B=\langle \A, {\sf d}_{ij}\rangle_{i,j<
 \alpha}$$ where $\A$ is a $\QA_{\alpha}$ (i.e. it satisfies all the equations above),  ${\sf d}_{ij}$ is a
constant  and the following equations hold, for all $i, j, k <  \alpha$:
	\begin{enumerate}
\item $\sub i j {\sf d}_{ij}=1$,
\item  $x\cdot {\sf d}_{ij}\leq \sub i j x$.
\end{enumerate}

\end{description}
\end{definition}

\begin{definition} Let $\alpha$ be an ordinal. We say that a variety $\sf V$ is a variety between $\Df_\alpha$ and $\QEA_\alpha$  
if the signature of $\sf V$ 
expands that of $\Df_\alpha$ and is contained in the signature of $\QEA_\alpha$. Furthermore, 
any  equation formulated in the signature of $\Df_\alpha$ that holds in $\sf V$ also holds in $\Sc_\alpha$ 
and all  equations that hold in $\sf V$ holds in $\sf QEA_\alpha$. 
\end{definition}
Proper examples include $\Sc$, $\CA_\alpha$ and ${\sf QA}_\alpha$ (meaning strictly between). 
Analogously we can define varieties between $\Sc_\alpha$ and $\CA_\alpha$ or $\QA_\alpha$ and $\QEA_\alpha$, and more generally between a class $\sf K$ of $\sf BAO$s and a generalized reduct of it.
Notions like neat reducts generalize verbatim to such 
algebras, namely, to $\Df$s and $\QEA$s, and in any variety in between. This stems from the observation that for any pair of ordinals $\alpha<\beta$, $\A\in \QEA_{\beta}$ and any non-Boolean exra operation in the signature of $\QEA_{\beta}$, $f$ say, 
if $x\in \A$ and $\Delta x\subseteq \alpha$, then $\Delta(f(x))\subseteq \alpha$. Here $\Delta x=\{i\in \beta: \c_ix\neq x\}$ is referred as {\it the dimension set} of $x$; it reflects algebraically the essentially free variables occuring in a formula $\phi$. A variable is essentially free in a formula $\Psi$ $\iff$ it is free in every formula equivalent to $\Psi$.\footnote{It can well happen that a variable is free in  formula that is equivalent to another formula in which this same variable is not free.}
Therefore given a variety $\V$ between $\Sc_{\beta}$ and $\QEA_{\beta}$, if $\B\in \V$ then the algebra $\Nrr_{\alpha}\B$ having universe $\{x\in \B: \Delta x\subseteq \alpha$\} is closed under all operations in the signature
of $\V$.   

\begin{definition} Let $2<n<\omega$. For a variety $\sf V$ between $\Df_n$ and $\QEA_n$, a {\it $\sf V$ set algebra} is a subalgebra of an algebra, having the same signature as $\V$, of the form $\langle \B({}^nU), f_i^U)$, say, 
where $f_i^U$ is identical to the interpretation 
of $f_i$ in the class of quasipolyadic equality set algebras.  Let $\A$ be an algebra having the same signature of $\V$; then $\A$ is {\it a representable $\sf V$ algebra}, or simply {\it representable} 
$\iff$ $\A$ is isomorphic to a subdirect product of $\sf V$ set algebras. We write $\sf RV$ for the class of representable $\V$ algebras
\end{definition}
It can be proved that the class $\sf RV$,  as defined above, is also closed under $\bold H$, so that it is a variety. This can be proved using the same argument to show that $\RCA_n$ is a variety, cf. Corollary \cite[3.1.77]{HMT2}. 
Take $\A\in \sf RV$, an ideal $J$ of $\A$, 
then show that $\A/J$ is in $\sf RV$. Ideals in $\sf BAO$s are defined as follows. We consider only $\sf BAO$s with extra unary non-Boolean operators to simplify notation. If $\A$ is a $\sf BAO$, 
then $J\subseteq \A$ is an ideal in $J$ if is a Boolean ideal and for any extra 
non-Boolean operator $f$, say, in the signature of $\sf BAO$, and $x\in \A$, $f(x)\in \A$; the quotient algebra $\A/J$ is defined the usual way since ideals defined in this way correspond to congruence relations defined on $\A$. 
\begin{theorem} \label{j} Let $2<n<\omega$. Let $\V$ be a variety between $\Df_n$ and $\QEA_n$. Then $\sf RV$ is not a finitely axiomatizable variety.
\end{theorem}
\begin{proof} In \cite{j} a sequence $\langle \A_i: i\in \omega\rangle$  of algebras is constructed such that $\A_i\in  \QEA_n$ and $\Rd_{df}\A_n\notin {\sf RDf}_n$, but $\Pi_{i\in \omega}\A_i/F\in {\sf RQEA}_n$ for any non principal ultrafilter on $\omega$.
An appilcation of Los' Theorem, taking the ultraproduct of  $\V$ reduct of the $\A_i$s,  finishes the proof. In more detail, let $\Rd_{\V}$ denote restricting the signature to that of $\V$. 
Then  $\Rd_{\V}\A_i\notin \sf RV$ and $\Rd_{\V}\Pi_{i\in I}(\A_i/F)\in \sf RV$.
\end{proof} 
The last result generalizes to infinite dimensions replacing finite axiomatization by axiomatized by a finite schema \cite{HMT2, t}. 
We consider relation algebras as algebras of the form ${\cal R}=\langle R, +, \cdot, -, 1', \smile, ; , \rangle$, where $\langle R, +, \cdot , -\rangle$ is a Boolean algebra $1'\in R$, $\smile$ is a unary operation and $;$ is a binary operation. 
A relation algebra is {\it representable}$\iff$ it is isomorphic to a subalgebra of the form $\langle \wp(X), \cup, \cap, \sim, \smile, \circ, Id\rangle$, where $X$ is an equivalence relation, $1'$ is interpreted as the identity relation, $\smile$ is the operation of forming converses, 
and$;$  is interpreted as composition of relations. 
Following standard notation, $(\sf R)RA$ denotes the class of (representable) relation algebras. The class $\sf RA$ is  a discriminator variety that is finitely axiomatizable, cf. \cite[Definition 3.8, Theorems 3.19]{HHbook}.  
We let $\sf CRRA$ and $\sf LRRA$, 
denote the classes of completely representable $\RA$s, and its elementary closure, namely, the class of $\RA$s 
satisfying the Lyndon conditions  as defined in \cite[\S 11.3.2]{HHbook}, respectively. Complete representability of $\sf RA$s is defined like the $\CA$ case.
All of the above classes of algebras are instances of $\sf BAO$s. 
The action of the non--Boolean operators in a completely additive (where operators distribute over arbitrary joins componentwise) 
atomic $\sf BAO$, is determined by their behavior over the atoms, and
this in turn is encoded by the {\it atom structure} of the algebra.

\section{Complete representability via atomic dilations}

We recall from \cite[Definition~5.4.16]{HMT2},  the notion of neat reducts of polyadic algebras. We shall be dealing with infinite dimensional such
algebras. Because infinite cylindrfication is allowed,  the definition of neat reducts is different from the $\CA$ case. 
We define the neat reduct operator 
only $\PA$s; the $\PEA$ case is entirely analgous considering diagonal elements.
 \begin{definition} Let $J\subseteq \beta$ and
$\A=\langle A,+,\cdot ,-, 0, 1,{\sf c}_{(\Gamma)}, {\sf s}_{\tau} \rangle_{\Gamma\subseteq \beta ,\tau\in {}^{\beta}\beta}$
be a $\PA_{\beta}$.
Let $Nr_J\B=\{a\in A: {\sf c}_{(\beta\sim J)}a=a\}$. Then
$${\bf Nr}_J\B=\langle Nr_{J}\B, +, \cdot, -, {\sf c}_{(\Gamma)}, {\sf s}'_{\tau}\rangle_{\Gamma\subseteq J,  \tau\in {}^{\alpha}\alpha}$$
where ${\sf s}'_{\tau}={\sf s}_{\bar{\tau}}.$ Here $\bar{\tau}=\tau\cup Id_{\beta\sim \alpha}$.
The structure ${\bf Nr}_J\B$ is an algebra, called the {\it $J$--compression} of $\B$.
When $J=\alpha$, $\alpha$ an ordinal, then ${\bf Nr}_{\alpha}\B\in \PA_{\alpha}$ and it is
called the neat $\alpha$ reduct of $\B$. 
\end{definition}
Assume that $\B\in \PEA_{\beta}$ for some infinite ordinal $\beta$. Then for $n<\omega$, ${\bf Nr}_n\B\subseteq \Nr_n\Rd_{qea}\B$, where 
$\Rd_{qea}$ denotes the quasi--polyadic reduct of $\B$, obtained by discarding infinitary
substitutions and  the definition of the neat reduct opeartor $\Nr_n$ here is like the $\CA$ case not involving infinitary cylindrifiers.
Indeed, if $x\in {\bf Nr}_n\B$, then ${\sf c}_{(\beta\setminus n)}x=x$, so for any $i\in \beta\setminus n$,
${\sf c}_ix\leq {\sf c}_{(\beta\setminus n)}x=x\leq {\sf c}_ix$,
hence ${\sf c}_ix=x$.   However, the converse might not be true. If $x\in \Nr_n\Rd_{qea}\B$, then ${\sf c}_ix=x$ for all $i\in \beta\setminus n$, but 
this does not imply that ${\sf c}_{(\beta\setminus n)}x=x;$ it can happen that ${\sf c}_{(\beta\setminus n)}x>x={\sf c}_ix$ (for all $i\in \beta\setminus n$).
We will show in a moment that if 
$\C\in {\sf PEA}_{\omega}$ is atomic and $n<\omega$, then both $\Nr_n\Rd_{qea}\C$ and ${\bf Nr}_n\C$ are 
completely representable $\PEA_n$s.  
This gives a plethora of completely representable ${\sf PEA}_n$s whose $\CA$ reducts are (of course) 
also completely representable.

For $\alpha\geq \omega$, we let ${\sf CPA}_{\alpha}$ (${\sf CPEA}_{\alpha}$) denote the reduct of $\PA_{\alpha}$($\PEA_{\alpha}$) whose signature is 
obained from that of $\PA_{\alpha}$ ($\PEA_{\alpha}$) by discarding all infinitary cylindrifiers, and its axiomatization is that of 
$\PA_{\alpha}$($\PEA_{\alpha})$ restricted to the new signature.
$\QA_{\alpha}$ ($\QEA_{\alpha}$) denotes the class of quasi--polyadic (equality) algebras obtained by restricting the signature and axiomatization of $\PA_{\alpha}(\PEA_{\alpha}$) to
only finite substitutions and cylindrifiers.  So here the signature does not contain {\it infinitary} substitutions, the $\s_{\tau}$s are defined only for 
those maps $\tau:\alpha\to \alpha$ that move 
only finitely many points.   With cylindrifiers defined only on finitely many indices, the neat reduct operator $\Nr$ for $\QA_{\alpha}$, $\QEA_{\alpha}$, $\sf CPA_{\alpha}$ and 
$\sf CPEA_{\alpha}$ is defined analogous 
to the $\CA$ case. 
We present analogous positive results typically of the form:  If $\bold K$ is any of the classes defined above 
(like $\CPEA, CPA, QEA, QA$), $\D\in \bold K_{\omega}$ is atomic, $n<\omega$, 
then (under certain conditions on $\D$)  $\Nr_n\D$ is completely representable. The `certain conditions' will be 
formulated only for the dilation $\D$ and will not depend on $n$. For example for 
$\PEA$, mere atomicity  of $\D$ will suffice, 
for $\PA$ we will need complete additivity of $\D$ too. 

We need a crucial lemma. But first a definition:
\begin{definition} A {\it transformation system} is a quadruple of the form $(\A, I, G, {\sf S})$ where $\A$ is an 
algebra of any signature, 
$I$ is a non--empty set (we will only be concerned with infinite sets),
$G$ is a subsemigroup of $(^II,\circ)$ (the operation $\circ$ denotes composition of maps) 
and ${\sf S}$ is a homomorphism from $G$ to the semigroup of endomorphisms of $\A$. 
Elements of $G$ are called transformations. 
\end{definition}
The next lemma says that, roughly, in the presence of all substitution operators in the infinite dimensional case, one can 
form dilations in any higher dimension.
\begin{lemma}\label{dilation} Let $\alpha$ be an infinite ordinal and $\bold K\in \{\PA, \PEA\}$. Let $\D\in \bold K_{\alpha}.$ 
Then for any ordinal $\mathfrak{n}>\alpha$,  there exists $\B\in \bold K_{\mathfrak{n}}$ 
such that $\D={\bf Nr}_{\alpha}\B$. 
Furthermore, if $\D$ is atomic (complete), then $\B$ can be chosen to be atomic (complete). An entirely analogous result holds for 
$\sf CPA$ and $\sf CPEA$ replacing the operator ${\bf Nr}$ by the neat reduct operator $\Nr$.
\end{lemma}
\begin{proof} 
Let $\bold K\in \{\sf PA, PEA, CPA, CPEA\}.$ Assume that $\D\in \bold K_{\alpha}$ and that $\mathfrak{n}>\alpha$. 
If $|\alpha|=|\mathfrak{n}|$, then one fixes a bijection $\rho:\mathfrak{n}\to \alpha$, and defines the $\mathfrak{n}$-dimensional dilation of the diagonal free
reduct of $\D$, having the same universe as $\D$, by re-shuffling the operations of $\D$ along $\rho$ \cite{DM}. 
Then one defines diagonal elements in the $\mathfrak{n}$-dimensional dilation of the diagonal free reduct of $\D$, 
by using the diagonal elements in  $\D$ \cite[Theorem 5.4.17]{HMT2}. 
Now assume that $|\mathfrak{n}|>|\alpha|$.  
Let ${\sf End}(\D)$ be the semigroup of Boolean endomorphisms on
$\D$.  Then the map $\sf S:{}^\alpha\alpha\to {\sf End}(\A)$ defined  via $\tau\mapsto {\sf s}_{\tau}$ is a homomorphism of semigroups.
The operation on both semigroups is composition of maps, so that $(\D, \alpha, {}^{\alpha}\alpha, \sf S)$ is a transformation system. 
For any set $X$, let $F(^{\alpha}X,\A)$
be the set of all maps from $^{\alpha}X$ to $\A$ endowed with Boolean  operations defined pointwise and for
$\tau\in {}^\alpha\alpha$ and $f\in F(^{\alpha}X, \A)$, put ${\sf s}_{\tau}f(x)=f(x\circ \tau)$.
This turns $F(^{\alpha}X,\A)$ to a transformation system as well.
The map $H:\A\to F(^{\alpha}\alpha, \A)$ defined by $H(p)(x)={\sf s}_xp$ is
easily checked to be an embedding of transfomation systems. Assume that $\beta\supseteq \alpha$. Then $K:F(^{\alpha}\alpha, \A)\to F(^{\beta}\alpha, \A)$
defined by $K(f)x=f(x\upharpoonright \alpha)$ is an embedding, too.
These facts are fairly straightforward to establish
\cite[Theorems 3.1, 3.2]{DM}.
Call $F(^{\beta}\alpha, \D)$ a minimal functional dilation of $F(^{\alpha}\alpha, \D)$.
Elements of the big algebra, or the (cylindrifier free)
functional dilation, are of form ${\sf s}_{\sigma}p$,
$p\in F(^{\alpha}\alpha, \D)$ where $\sigma\upharpoonright \alpha$ is injective \cite[Theorems 4.3-4.4]{DM}.
Let $\B^{-c, -d}=F({}^{\mathfrak{n}}\alpha, \D).$
Let 
$\rho$ is any permutation such that $\rho\circ \sigma(\alpha)\subseteq \sigma(\alpha).$
For the $\PA$ case one defines cylindrifiers on $\B^{-c, -d}$ by setting for each $\Gamma\subseteq \mathfrak{n}:$
$${\sf c}_{(\Gamma)}{\sf s}_{\sigma}^{(\B^{-c, -d})}p={\sf s}_{\rho^{-1}}^{(\B^{-c, -d})} {\sf c}_{\rho(\{(\Gamma)\}\cap \sigma \alpha)}^{\D}{\sf s}_{(\rho\sigma\upharpoonright \alpha)}^{\D}p.$$
For the cases $\sf CPA$ case, one defines cylindrifiers on $\B^{-c, -d}$ by restricting $\Gamma$ to singletons, setting for each  $i\in  \mathfrak{n}:$
$${\sf c}_{i}{\sf s}_{\sigma}^{(\B^{-c, -d})}p={\sf s}_{\rho^{-1}}^{(\B^{-c, -d})} {\sf c}_{(\rho(i)\cap \sigma(\alpha))}^{\D}
{\sf s}_{(\rho\sigma\upharpoonright \alpha)}^{\D}p.$$
In both cases, the definition is sound, that is, it is independent of $\rho, \sigma, p$; furthermore, it agrees with the old cylindrifiers in $\D$.
Denote the resulting algebra by $\B^{-d}$.

When $\D\in \PA_{\alpha}$, identifying algebras with their transformation systems
we get that $\D\cong {\bf Nr}_{\alpha}\B^{-d}$, via the isomorphism $H$ defined
for $f\in \D$ and $x\in {}^{\mathfrak{n}}\alpha$ by,
$H(f)x=f(y)$ where $y\in {}^{\alpha}\alpha$ and $x\upharpoonright \alpha=y$,
\cite[Theorem 3.10]{DM}. In \cite[Theorems 4.3. 4.4]{DM} 
it is shown that $H(\D)={\bf Nr}_{\alpha}\B^{-d}$ where $\B^{-d}=\{\s_{\sigma}^{(\B^{-d})}p: p\in \D: \sigma\upharpoonright \alpha \text { is injective}\}.$
 When $\D\in {\sf CPA}_{\alpha}$, identifying $\D$ with $H(\D)$, where $H$ is defined like in the $\PA$ case, 
we get that $\D\subseteq \Nr_{\alpha}\B^{-d}$ with $\B^{-d}=\{\s_{\sigma}^{(\B^{-d})}p: p\in \D: \sigma\upharpoonright \alpha \text { is injective}\}.$
We show that $\Nr_{\alpha}\B^{-d}\subseteq \D$, so that $\D=\Nr_{\alpha}\B^{-d}$. 
Let $x\in \Nr_{\alpha}\B^{-d}$. Then there exist $y\in \D$ and $\sigma:\beta\to \beta$ with $\sigma\upharpoonright \alpha$ injective, such that 
$x=\s_{\sigma}^{(\B^{-d})}y$. Choose any $\tau:\beta\to \beta$ such that 
$\tau(i)=i$ for all $i\in \alpha$ and $(\tau\circ \sigma)(i)\in \alpha$ for all $i\in \alpha$. Such a $\tau$ clearly exists. Since $\Delta x\subseteq \alpha$, 
and $\tau$ fixes $\alpha$ pointwise, we have $\s_{\tau}^{(\B^{-d})}x=x$.
Then $x=\s_{\tau}^{(\B^{-d})}x=\s_{\tau}^{(\B^{-d})}\s_{\sigma}^{(\B^{-d})}y=\s_{\tau\circ \sigma}^{(\B^{-d})}y= \s_{\tau\circ \sigma\upharpoonright \alpha}^{\D}y\in \D$.
In all cases, having at hand $\B^{-d}$, for all 
$i<j<\mathfrak{n}$,  the diagonal element ${\sf d}_{ij}$ (in $\B^{-d}$) can be defined, using
the diagonal elements in $\D$, as in
\cite[Theorem 5.4.17]{HMT2}, obtaining the expanded required structure $\B$.  
The expanded structure $\B$ has Boolean reduct isomorphic to $F({}^\mathfrak{n}\alpha, \D)$. 
In particular, $\B$ is atomic (complete) if 
$\D$ is atomic (complete), 
because a product of an atomic (complete) Boolean algebras is atomic (complete).
\end{proof}
The proof of the following lemma follows from the definitions.
\begin{lemma}\label{join2} If $\A$, $\B$ and $\C$ are Boolean algebras, such that 
$\A\subseteq \B\subseteq \C$, $\B\subseteq_c \C$ and $\A\subseteq_c \C$, 
then $\A\subseteq_c \B$. 
\end{lemma}
For simplicity of notation, if $\beta\geq \omega$, $\B\in \PA_{\beta}(\PEA_{\beta})$, and $n<\omega$,  then we write $\Nr_n\B$ 
for $\Nr_n\Rd_{qa}\B$ $(\Nr_n\Rd_{qea}\B)$, where $\Rd_{qa}$ denotes `quasi-polyadic reduct'.

{\it In this section we understand complete representability for $\alpha$--dimensional algebras, $\alpha$ any ordinal,
in the classical sense with respect to generalized cartesian $\alpha$--dimensional spaces.}

It is shown in in \cite{au}, that for any infinite ordinal $\alpha$, if  $\A\in \PA_{\alpha}$ is atomic and completely additive, then it is completely representable. 
From this it can be concluded that for any $n<\omega$, 
any complete subalgebra of $\Nr_n\D$ is completely representable using lemma \ref{join}, 
because $\Nr_n\D\subseteq_c \D$ (as can be easily distilled from the next proof).
The result in \cite{au} holds for $\sf CPA$'s, cf. theorem \ref{paaa}, but it does not hold for $\PEA_{\omega}$s and $\sf CPEA_{\omega}$s.
It is not hard to construct atomic algebras in the last two classes that are not even representable, let alone completely representable. 
But for such (non--representable) algebras the $n$--neat reduct, for any $n<\omega$, 
will be completely representable 
as proved next (in theorems \ref{pa} and \ref{paa}):
\begin{theorem}\label{pa} If $2<n<\omega$ and $\D\in \PEA_{\omega}$ is atomic, then
any complete subalgebra of $\Nr_n\D$ is completely representable. In particular, 
${\bf Nr}_n\D$ is completely representable. 
\end{theorem}\begin{proof}
We identify notationally set algebras with their domain. 
Assume that  $\A\subseteq_c{\Nr}_n\D$, where $\D\in \PEA_{\omega}$ is atomic. 
We want to completely represent $\A$.  Let $c\in \A$ be non--zero. We will find  a homomorphism $f:\A\to \wp(^nU)$ 
such that $f(c)\neq 0$, and $\bigcup_{y\in Y}f(y)={}^nU$, whenever $Y\subseteq \A$ satisfies $\sum^{\A}Y=1$.
Assume for the moment  (to be proved in a while) that $\A\subseteq_c \D$. Then by lemma \ref{join} $\A$ is atomic, because $\D$. For brevity, let $X=\At\A$. 
Let $\mathfrak{m}$ be the local degree of $\D$, $\mathfrak{c}$ its effective cardinality 
and let $\beta$ be any cardinal such that $\beta\geq \mathfrak{c}$
and $\sum_{s<\mathfrak{m}}\beta^s=\beta$; such notions are defined in \cite{DM, au}.
We can assume by lemma \ref{dilation}, that $\D=\Nr_{\omega}\B$, with $\B\in \PEA_{\beta}$.  
For any ordinal $\mu\in \beta$, and $\tau\in {}^{\mu}\beta$,  write $\tau^+$ for $\tau\cup Id_{\beta\setminus \mu}(\in {}^\beta\beta$).
Consider the following family of joins evaluated in $\B$,
where $p\in \D$, $\Gamma\subseteq \beta$ and
$\tau\in {}^{\omega}\beta$:
(*) $ {\sf c}_{(\Gamma)}p=\sum^{\B}\{{\sf s}_{{\tau^+}}p: \tau\in {}^{\omega}\beta,\ \  \tau\upharpoonright \omega\setminus\Gamma=Id\},$ and (**):
$\sum {\sf s}_{{\tau^+}}^{\B}X=1.$
The first family of joins exists \cite[Proof of Theorem 6.1]{DM}, \cite{au}, and the second exists, 
because $\sum ^{\A}X=\sum ^{\D}X=\sum ^{\B}X=1$ and $\tau^+$ is completely additive, since
$\B\in \PEA_{\beta}$. 
The last equality of suprema follows from the fact that $\D=\Nr_{\omega}\B\subseteq_c \B$ and the first
from the fact that $\A\subseteq_c \D$. We prove the former, the latter is exactly the same replacing
$\omega$ and $\beta$, by $n$ and $\omega$, respectivey, proving that $\Nr_n\D\subseteq_c \D$, hence $\A\subseteq_c \D$.  
We prove that $\Nr_{\omega}\B\subseteq_c \B$. Assume that $S\subseteq \D$ and $\sum ^{\D}S=1$, and for contradiction, that there exists $d\in \B$ such that
$s\leq d< 1$ for all $s\in S$. Let  $J=\Delta d\setminus \omega$ and take  $t=-{\sf c}_{(J)}(-d)\in {\D}$.
Then  ${\sf c}_{(\beta\setminus \omega)}t={\sf c}_{(\beta\setminus \omega)}(-{\sf c}_{(J)} (-d))
=  {\sf c}_{(\beta\setminus \omega)}-{\sf c}_{(J)} (-d)
=  {\sf c}_{(\beta\setminus \omega)} -{\sf c}_{(\beta\setminus \omega)}{\sf c}_{(J)}( -d)
= -{\sf c}_{(\beta\setminus \omega)}{\sf c}_{(J)}( -d)
=-{\sf c}_{(J)}( -d)
=t.$
We have proved that $t\in \D$.
We now show that $s\leq t<1$ for all $s\in S$, which contradicts $\sum^{\D}S=1$.
If $s\in S$, we show that $s\leq t$. By $s\leq d$, we have  $s\cdot -d=0$.
Hence by ${\sf c}_{(J)}s=s$, we get $0={\sf c}_{(J)}(s\cdot -d)=s\cdot {\sf c}_{(J)}(-d)$, so
$s\leq -{\sf c}_{(J)}(-d)$.  It follows that $s\leq t$ as required. Assume for contradiction that 
$1=-{\sf c}_{(J)}(-d)$. Then ${\sf c}_{(J)}(-d)=0$, so $-d =0$ which contradicts that $d<1$. We have proved that $\sum^{\B}S=1$,
so $\D\subseteq_c \B$.
Let $F$ be any Boolean ultrafilter of $\B$ generated by an atom below $a$. We show that $F$
will preserve the family of joins in (*) and (**).
We use a simple topological argument  used by the author in \cite{au}. 
One forms nowhere dense sets in the Stone space of $\B$ 
corresponding to the aforementioned family of joins 
as follows:
The Stone space of (the Boolean reduct of) $\B$ has underlying set,  the set of all Boolean ultrafilters
of $\B$. For $b\in \B$, let $N_b$ be the clopen set $\{F\in S: b\in F\}$.
The required nowhere dense sets are defined for $\Gamma\subseteq \beta$, $p\in \D$ and $\tau\in {}^{\omega}\beta$ via:
$A_{\Gamma,p}=N_{{\sf c}_{(\Gamma)}p}\setminus \bigcup_{\tau:\omega\to \beta}N_{{\sf s}_{\tau^+}p}$, 
and $A_{\tau}=S\setminus \bigcup_{x\in X}N_{{\sf s}_{\tau^+}x}.$
The principal ultrafilters are isolated points in the Stone topology, so they lie outside the nowhere dense sets defined above.
Hence any such ultrafilter preserve the joins in (*) and (**). 
Fix a principal ultrafilter $F$ preserving (*) and (**) with $a\in F$. 
For $i, j\in \beta$, set $iEj\iff {\sf d}_{ij}^{\B}\in F$.
Then by the equational properties of diagonal elements and properties of filters, it is easy to show that $E$ is an equivalence relation on $\beta$.
Define $f: \A\to \wp({}^n(\beta/E))$, via $x\mapsto \{\bar{t}\in {}^n(\beta/E): {\sf s}_{t\cup Id_{\beta\sim n}}^{\B}x\in F\},$
where $\bar{t}(i/E)=t(i)$ ($i<n$) and $t\in {}^n\beta$. 
We show that  $f$ is a well--defined homomorphism (from (*)) and that $f$ is complete 
such that $f(c)\neq 0$. The last follows by observing that $Id\in f(c)$.
Let $V={}^\beta\beta^{(Id)}$. To show that $f$ is well defined, it suffices to show  that for all $\sigma, \tau\in V$, 
if $(\tau(i), \sigma(i))\in E$ for all $i\in \beta$,   then 
for any $x\in \A$, ${\sf s}_{\tau}x\in F\iff {\sf s}_{\sigma}x\in F.$  
We proceed by by induction on
$|\{i\in \beta: \tau(i)\neq \sigma(i)\}| (<\omega)$.  
If $J=\{i\in \beta: \tau(i)\neq \sigma(i)\}$ is empty, the result is obvious. 
Otherwise assume that $k\in J$. 
We introduce a helpful piece of notation.
For $\eta\in V$, let  
$\eta(k\mapsto l)$ stand for the $\eta'$ that is the same as $\eta$ except
that $\eta'(k)=l.$ 
Now take any 
$\lambda\in  \{\eta\in \beta: (\sigma){^{-1}}\{\eta\}= (\tau){^{-1}}\{\eta\}=
\{\eta\}\}\smallsetminus \Delta x.$
Recall that $\Delta x=\{i\in \beta: {\sf c}_ix\neq x\}$ and that $\beta\setminus \Delta x$
is infinite because $\Delta x\subseteq n$, so such a $\lambda$ exists. Now 
we freely use properties of substitutions for cylindric algebras.
We have by \cite[1.11.11(i)(iv)]{HMT2}
(a) ${\sf s}_{\sigma}x={\sf s}_{\sigma k}^{\lambda}{\sf s}_{\sigma(k\mapsto \lambda)}x,$
and (b)
${\sf s}_{\tau k}^{\lambda}({\sf d}_{\lambda, \sigma k}\cdot  {\sf s}_{\sigma} x)
={\sf d}_{\tau k, \sigma k} {\sf s}_{\sigma} x,$
and (c)
${\sf s}_{\tau k}^{\lambda}({\sf d}_{\lambda, \sigma k}\cdot {\sf s}_{\sigma(k\mapsto \lambda)}x)
= {\sf d}_{\tau k,  \sigma k}\cdot {\sf s}_{\sigma(k\mapsto \tau k)}x,$
and finally (d)
${\sf d}_{\lambda, \sigma k}\cdot {\sf s}_{\sigma k}^{\lambda}{\sf s}_{{\sigma}(k\mapsto \lambda)}x=
{\sf d}_{\lambda, \sigma k}\cdot {\sf s}_{{\sigma}(k\mapsto \lambda)}x.$
Then by (b), (a), (d) and (c), we get,
\begin{align*}
{\sf d}_{\tau k, \sigma k}\cdot {\sf s}_{\sigma} x
&=  {\sf s}_{\tau k}^{\lambda}({\sf d}_{\lambda,\sigma k}\cdot {\sf s}_{\sigma}x)\\
&={\sf s}_{\tau k}^{\lambda}({\sf d}_{\lambda, \sigma k}\cdot {\sf s}_{\sigma k}^{\lambda}
{\sf s}_{{\sigma}(k\mapsto \lambda)}x)\\
&={\sf s}_{\tau k}^{\lambda}({\sf d}_{\lambda, \sigma k}\cdot {\sf s}_{{\sigma}(k\mapsto \lambda)}x)\\
&= {\sf d}_{\tau k,  \sigma k}\cdot {\sf s}_{\sigma(k\mapsto \tau k)}x.
\end{align*}
But $F$ is a filter and $(\tau k,\sigma k)\in E$, we conclude that
$${\sf s}_{\sigma}x\in F \iff{\sf s}_{\sigma(k\mapsto \tau k)}x\in F.$$
The conclusion follows from the induction hypothesis.
We have proved that $f$ is well defined.
We now check that $f$ is a homomorphism, i.e. it 
preserves the operations.
For $\sigma\in {}^{n}\beta$, recall that 
$\sigma^+$ denotes $\sigma \cup Id_{{\beta}\sim n}\in {}^{\beta}\beta^{(Id)}.$ 
\begin{itemize}
\item Boolean operations:  Since $F$ is maximal we have  
$\bar{\sigma}\in f(x+y)\iff {\sf s}_{\sigma^+}(x+y)\in F\iff {\sf s}_{\sigma^+}x+{\sf s}_{\sigma^+}y\in F\iff 
{\sf s}_{\sigma^+} x \text { or } {\sf s}_{\sigma^+} y\in F\iff  
\bar{\sigma}\in f(x)\cup f(y).$
We now check complementation.
$$\bar{\sigma} \in f(-x) \iff {\sf s}_{\sigma^+}(-x)\in F \iff-{\sf s}_{\sigma^+x}\in F 
\iff {\sf s}_{\sigma^+} x\notin F \iff \bar{\sigma}\notin f(x).$$

\item Diagonal elements: Let $k,l<n$. Then we have:   
$\sigma\in f{\sf d}_{kl}\iff {\sf s}_{\sigma^{+}}{\sf d}_{kl}\in F \iff
{\sf d}_{\sigma k, \sigma l}\in F
\iff (\sigma k, \sigma l)\in E \iff
\sigma k/E=\sigma l/E \iff 
\bar\sigma(k)=\bar\sigma(l)
\iff\bar{\sigma} \in {\sf d}_{kl}.$

\item Cylindrifications: Let $k<n$ and $a\in A$.  Let $\bar{\sigma}\in {\sf c}_kf(a)$. 
Then for some $\lambda\in \beta$, we have
$\bar{\sigma}(k \mapsto  \lambda/E)\in f(a)$ 
hence 
${\sf s}_{\sigma^+(k\mapsto \lambda)}a\in F$. 
It follows from  the inclusion $a\leq {\sf c}_ka$ that 
${\sf s}_{\sigma^+(k\mapsto \lambda)}{\sf c}_ka \in F$,
so ${\sf s}_{\sigma^+}{\sf c}_ka\in F.$
Thus ${\sf c}_kf(a)\subseteq f({\sf c}_ka.)$

We prove the other more difficult inclusion that uses the condition (*) of eliminating cylindrifiers.
Let $a\in A$ and $k<n$. Let $\bar\sigma'\in f{\sf c}_ka$ and let $\sigma=\sigma'\cup Id_{\beta\sim n}$. Then 
${\sf s}_{\sigma}^{\B}{\sf c}_ka={\sf s}_{\sigma'}^{\B}{\sf c}_ka\in F.$
Pick $\lambda\in \{\eta \in \beta:  \sigma^{-1}\{\eta\}=\{\eta\}\}\smallsetminus \Delta a$, such a $\lambda$ exists because $\Delta a$ is finite, and 
$|\{i\in \beta:\sigma(i)\neq i\}|<\omega.$
Let $\tau=\sigma\upharpoonright \mathfrak{n}\smallsetminus \{k,\lambda\}\cup \{
(k,\lambda), (\lambda,k)\}.$
Then (in $\B$):
$${\sf c}_{\lambda}{\sf s}_{\tau}a={\sf s}_{\tau}{\sf c}_ka={\sf s}_{\sigma}{\sf c}_ka\in F.$$
By the construction of $F$, there is some $u(\notin \Delta({\sf s}_{\tau}^{\B}a))$  
such that ${\sf s}_{u}^{\lambda}{\sf s}_{\tau}a\in F,$ so ${\sf s}_{\sigma(k\mapsto u)}a\in F.$
Hence $\sigma(k\mapsto u)\in f(a),$ from which we 
get that  $\bar{\sigma}'\in {\sf c}_k f(a)$.

\item Substitutions: Direct since substitution operations are Boolean endomorphisms

\end{itemize}

We show that the non--zero homomorphism 
$f$ is an atomic, hence, a complete representation.
 By construction, for every $s\in {}^{n}(\beta/E)$, 
there exists $x\in X(=\At\A)$, such that ${\sf s}_{s\cup Id_{\beta\sim n}}^{\B}x\in F$, 
from which we get the required, namely, that  
$\bigcup_{x\in X}f(x)={}^n(\beta/E).$
The complete representability of ${\bf Nr}_n\D$ follows from lemmata, \ref{join}, 
\ref{join2}, by observing that  ${\bf Nr}_n\D\subseteq_c \D$, hence ${\bf Nr}_n\D\subseteq_c \Nr_n\D$. 
\end{proof}
For $\CPEA$s, we have a slightly weaker result:
\begin{theorem}\label{paa} If $n<\omega$ and  $\D\in {\sf CPEA}_{\omega}$ is atomic,  then any complete subalgebra 
of $\Nr_n\Cm\At\D$  is completely representable. In particular, if $\D$ is complete and atomic, then $\Nr_n\D$ is completely representable.
\end{theorem}
\begin{proof}
Let $\D\in {\sf CPEA}_{\omega}$ be atomic. Let $\D^*=\Cm\At\D$. Then $\D^*$ is complete and atomic and $\Nr_n\D^*\subseteq_c \D^*$.
To prove the last $\subseteq_c$, assume for contradiction that 
there is some $S\subseteq \Nr_{n}\D^*$,  $\sum ^{\Nr_{n}\D^*}S=1$, and there exists $d\in \D^*$ such that
$s\leq d< 1$ for all $s\in S$. Take $t=-\bigwedge_{i\in \omega\setminus n}(-{\sf c}_i-)d$.
This infimum is well defined because $\D^*$ is complete. Like in the previous proof it can be proved that ${\sf c}_it=t$ for all $i\in \omega\setminus n$,
hence $t\in \Nr_{n}\D^*$
 and that $s\leq t<1$ for all $s\in S$, 
which contradicts that $\sum ^{\Nr_{n}\D^*}S=1$.
Let $\beta$ be a regular cardinal $>|\D^*|$ and by lemma \ref{dilation}, 
let $\B\in {\sf CPEA}_{\beta}$ be   complete and atomic such that $\D^*=\Nr_{\omega}\B$. 
Then we have the following chain of complete embeddings:
 $\Nr_n\D^*\subseteq_c \D^*=\Nr_{\omega}{\B}\subseteq_c \B$; the last $\subseteq_c$ follows like above using that $\B$ is complete.
From the first $\subseteq_c$, since $\D^*$ is atomic, we get by lemma \ref{join}, 
that $\Nr_n\D^*$ is atomic. Let $X=\At\Nr_n\D^*$. 
Then also from the first $\subseteq_c$, we get that $\sum^{\D^*}X=1$, so $\sum^{\B}X=1$ because $\D^*\subseteq_c \B$. 
For $k<\beta$, $x\in \D^*$ and $\tau\in {}^\omega\beta$,
the following joins hold in $\B$: 
(*) ${\sf c}_kx=\sum_{l\in \beta}^{\B}{\sf s}_l^kx$ and (**) $\sum {\sf s}_{\tau^{+}}^{\B}X=1$, where 
$\tau^+=\tau\cup Id_{\beta\setminus \omega}(\in {}^\beta\beta)$. The join (**) holds, because $\s_{\tau^{+}}^{\B}$ is completely additive, since $\B$ is completely additive.
To prove (*), fix $k<\beta$. Then for all $l\in \beta$, we have ${\sf s}_l^kx\leq {\sf c}_kx$.
Conversely, assume that $y\in \B$ is an upper bound for $\{{\sf s}_l^kx: l\in \beta\}$. 
Let $l\in \beta\setminus (\Delta x\cup \Delta y)$; such an $l$ exists, because $|\Delta x|<\beta$, $|\Delta y|<\beta$ and $\beta$ is regular.
Hence, we get that ${\sf c}_lx=x$ and ${\sf c}_ly=y$.  But then ${\sf c}_l{\sf s}_l^kx\leq y$, and so ${\sf c}_kx\leq y$.
We have proved that (*) hold. 
Let $\A=\Nr_n\D^*$. Let  $a\in \A$ be non--zero. We want to find a complete representation $f:\A\to \wp(V)$ ($V$ a unit 
of a $\sf Gs_{n}$, i.e a disjoint 
union of cartesian spaces)  such that $f(a)\neq 0$.  
Let $F$ be any Boolean ultrafilter of $\B$ generated by an atom below $a$. 
Then, like in the proof of theorem \ref{pa}, $F$
will preserve the family of joins in (*) and (**).
Next we proceed exactly like in the proof of theorem \ref{pa}. 
For $i, j\in \beta$, set $iEj\iff {\sf d}_{ij}^{\B}\in F$.
Then $E$ is an equivalence relation on $\beta$.
Define $f: \A\to \wp({}^n(\beta/E))$, via $x\mapsto \{\bar{t}\in {}^n(\beta/E): {\sf s}_{t\cup Id}^{\B}x\in F\},$
where $\bar{t}(i/E)=t(i)$ and $t\in {}^n\beta$. 
Then $f$ is well--defined, a homomorphism (from (*)) and atomic (from (**)).  Also
$f(a)\neq 0$ because $a\in F$, so $Id\in f(a)$.
\end{proof}

\begin{theorem}\label{paaa} If $\D\in \sf CPA_{\omega}$ is atomic and completely additive, then it is completely representable
\end{theorem}
\begin{proof} Replace $\D$ by its \de\ completion $\D^*=\Cm\At\D$. Then $\D$ is completely representable $\iff$ $\D^*$ is completely representable 
and furthermore $\D^*$ is complete.
It suffices thus to show that $\D^*$  is completely representable. 
One forms an atomic complete dilation $\B$ of $\D^*$ to a {\it regular cardinal $\beta>|\D^*|$} exactly as in lemma \ref{dilation}. 
For $\tau\in {}^{\omega}\beta$, let $\tau^+=\tau\cup Id_{\beta\setminus \omega}$.
Then like before $\D^*\subseteq_c \B$ and so    
the following family  of joins hold in $\B$: For all $i<\beta$, 
$b\in \B$  ${\sf c}_ib=\sum_{j\in \beta} {\sf s}_j^ix$
and for all $\tau\in {}^{\omega}\beta$, $\sum {\sf s}_{{\tau}^+}\At\D^*=1$. Let $a\in \D^*$ be non zero. 
Take any ultrafilter $F$ in the Stone space of $\B$ generated by an atom below $a$. 
Then   $f:\D^*\to \wp(^{\omega}\beta)$ defined 
via $d\mapsto \{\tau\in {}^{\omega}\beta: \s_{\tau^+}^{\B}d\in F\}$ 
is a complete representation of $\D^*$ such that $f(a)\neq 0$. 
\end{proof}

If the dilations are in $\QEA_{\omega}$ (an $\omega$ dimensional quasi--polyadic equality algebra) we have a weaker result. We do not know whether the result proved for $\PEA_{\omega}$ holds
when the $\omega$--dilation is an atomic $\QEA_{\omega}$. Entirely analogous results hold if we replace $\QEA_{\omega}$ by $\QA_{\omega}$.

\begin{theorem} \label{suf} Let $n<\omega$. Let $\D\in \QEA_{\omega}$ be atomic. 
Assume that  for all $x\in \D$ for all $k<\omega$, 
${\sf c}_kx=\sum_{l\in \omega}{\sf s}_l^kx$. 
If $\A\subseteq \Nr_{n}\D$  such that $\A\subseteq_c \D$ (this is stronger than $\A\subseteq_c \Nr_n\D$), 
then $\A$ is completely representable. 
\end{theorem}
\begin{proof} First observe that $\A$ is atomic, because $\D$ is atomic and $\A\subseteq_c \D$.
Accordingly, let $X=\At\A$. Let $a\in \A$ be non-zero. 
Like  before, one finds a principal ultrafilter $F$ such that $a\in F$ and $F$ preserves the family of joins ${\sf c}_{i}x=\sum^{\D}_{j\in \beta} {\sf s}_j^i x$, 
and $\sum {\sf s}^{\D}_{\tau}X=1$, where $\tau:\omega\to \omega$ is a finite transformation; that is $|\{i\in \omega: \tau(i)\neq i\}| <\omega$. 
The first family of joins exists by assumption, the second exists, since $\sum^{\D}X=1$ by 
$\A\subseteq_c \D$ and the $\s_{\tau}$s are completely additive.   Any principal ultrafilter $F$ generated by an atom below $a$ will do, as shown in the previous 
proof.  Again as before, the selected $F$ gives the required complete representation $f$ of $\A$. 
\end{proof}
The following example shows that the existence of the joins in theorem \ref{suf} is 
not necessary.

\begin{example}
Let $\D\in \QEA_{\omega}$ be the full weak set algebra with top element $^{\omega}\omega^{(\bold 0)}$ where $\bold 0$ is the constant $0$ sequence. 
Then, it is easy to show that for any $n<\omega$, $\Nr_n\A$
is completely representable.  
Let $X=\{\bold 0\}\in \D$. Then for all $i\in \omega$, we have ${\sf s}_i^0X=X$. 
But $(1, 0,\ldots )\in {\sf c}_0X$, so that $\sum_{i\in \omega} {\sf s}_i^0X=X\neq {\sf c}_0X$.  
Hence the joins in theorem \ref{suf} do not hold.
\end{example}

Now fix $1<n<\omega$ and let $\D$ be as in the previous example.  If we take $\D'=\Sg^{\D}\Nr_n\D$, then $\D'$ of course will still be a weak set algebra, and it will be locally finite, so that 
${\sf c}_ix=\sum^{\D}{\sf s}_j^ix$ for all $i<j<\omega$.
However, $\D'$ will be atomless as we proceed to show. Assume for contradiction  that it is not, and let $x\in \D'$ be an atom. 
Choose $k, l\in \omega$ with $k\neq l$ and ${\sf c}_kx=x$, this is possible since $\omega\setminus \Delta x$ is infinite. 
Then ${\sf c}_k(x\cdot  {\sf d}_{kl})=x$, so $x\cdot {\sf d}_{kl}\neq 0$. 
But $x$ is an atom, so $x\leq {\sf d}_{kl}$. This gives that $\Delta x=0$, 
and by \cite[Theorem 1.3.19]{HMT2} $x\leq -{\sf c}_k-{\sf d}_{kl}$. 
It is also easy to see that $({\sf c}_k -{\sf d}_{kl})^{\D'}={}^{\omega}\omega^{(\bold 0)}$, 
from which we conclude that $x=0$, 
which is a contradiction. 
For an ordinal $\alpha$, we let $\sf Gwsq_{\alpha}$ denote the class of $\QEA_{\alpha}$s  
whose cylindric reduct is a $\sf Gws_{\alpha}$, and the quasi-polyadic operations of substitutions 
defined like in quasi-polyadic equality set algebras relativized to $V$. That is, if $\A\in \sf Gwsq_{\alpha}$, then $\Rd_{ca}\A\in \sf Gws_{\alpha}$ with top element $V$ say, 
and for $X\in \A$, and $i< j<\alpha$,   ${\sf S}_{[i,j]}X=\{s\in V: s\circ [i, j]\in X\}$. 
\begin{theorem}\label{lastpa} Let  $\alpha$ be an infinite ordinal. 
\begin{enumarab}
\item If $\D\in \sf PEA_{\alpha+\omega}$ is atomic, then
any complete subalgebra of $\Nr_\alpha\D$ is completely representable with respet to $\sf Gwsq_{\alpha}$.  
\item If $\D\in \sf CPEA_{\alpha+\omega}$, then any complete subalgebra of $\Nr_{\alpha}\Rd_{qea}\Cm\At\D$ is completely representable with respect 
to $\sf Gwsq_{\alpha}$. In particular, if $\D$ is complete, then $\Nr_{\alpha}\Rd_{qea}\D$ is completely representable.
\end{enumarab}
\end{theorem} 
\begin{proof} The proof is like when $\alpha=n<\omega$. Let $\A\subseteq_c \Nr_{\alpha}\D$. We want to completely represent $\A$ with respect to a 
$\sf Gwsq_{\alpha}$. Given a non-zero $c\in \A$, one dilates $\D$ to $\B\in \PEA_{\beta}$, where $\beta$ is as specified in theorem \ref{pa}, 
and finds a principal ultrafilter $F$ generated by an atom below $c$ preserving the set of joins (*) and (**)
as stipulated in the proof of theorem \ref{pa}, replacing in these joins $\omega$ by the countable ordinal $\alpha+\omega$. 
In forming the required representation using $F$, the top element
will be a {\it weak space of dimension $\alpha$} and not `a cartesian square'. 
The map establishing the complete representation, is defined like before, but using the weak
space $^{\alpha}(\mathfrak{n}/E)^{(Id)}$, where $E$ is the equivalence relation 
defined as above on $\mathfrak{n}$ via $iEj \iff\ {\sf d}_{ij}^{\B}\in F$.
In more detail $f: \A\to \wp({}^\alpha(\beta/E)^{(Id)})$, via $x\mapsto \{\bar{t}\in {}^\alpha(\beta/E)^{(Id)}: {\sf s}_{t\cup Id_{\beta\sim \alpha}}^{\B}x\in F\},$
where $\bar{t}(i/E)=t(i)$ ($i<\alpha$) and $t\in {}^\alpha\beta$.  The $\CPEA$ case is entirely analogous. 
The proof is like the case when $\alpha=n<\omega$, dealt with in theorem \ref{paa}
replacing  once more 
set algebras by weak set algebras.
\end{proof}

\section{Finite dimensional algebras}
This section is devoted to showing that several classes of completely representable algebras (of relations) are not elementary. We need some preparing to do.
 From now on,  unless otherwise indicated, $n$ is fixed to be  a finite ordinal $>2$.
Let $i<n$. For $n$--ary sequences $\bar{x}$ and $\bar{y}$,  we write $\bar{x}\equiv_ i\bar{y}$ $\iff \bar{y}(j)=\bar{x}(j)$ for all $j\neq i$,
For $i, j<n$ the replacement $[i/j]$ is the map that is like the identity on $n$, except that $i$ is mapped to $j$ and the transposition 
$[i, j]$ is the like the identity on $n$, except that $i$ is swapped with $j$. 
\begin{definition}\label{sub} Let $m$ be a finite ordinal $>0$. An $\sf s$ word is a finite string of substitutions $({\sf s}_i^j)$ $(i, j<m)$,
a $\sf c$ word is a finite string of cylindrifications $({\sf c}_i), i<m$;
an $\sf sc$ word $w$, is a finite string of both, namely, of substitutions and cylindrifications.
An $\sf sc$ word
induces a partial map $\hat{w}:m\to m$:
\begin{itemize}

\item $\hat{\epsilon}=Id,$

\item $\widehat{w_j^i}=\hat{w}\circ [i|j],$

\item $\widehat{w{\sf c}_i}= \hat{w}\upharpoonright(m\smallsetminus \{i\}).$

\end{itemize}
If $\bar a\in {}^{<m-1}m$, we write ${\sf s}_{\bar a}$, or
${\sf s}_{a_0\ldots a_{k-1}}$, where $k=|\bar a|$,
for an  arbitrary chosen $\sf sc$ word $w$
such that $\hat{w}=\bar a.$
Such a $w$  exists by \cite[Definition~5.23 ~Lemma 13.29]{HHbook}.
\end{definition}
From now on,  unless otherwise indicated, $n$ is fixed to be  a finite ordinal $>2$.
\begin{definition}\label{game}
\begin{enumarab}
\item  Let $\K_n$ be any variety between $\Sc_n$ and $\QEA_n$. Assume that $\A\in \K_n$ is  atomic and that $m, k\leq \omega$. 
The {\it atomic game $G^m_k(\At\A)$, or simply $G^m_k$}, is the game played on atomic networks
of $\A$ using $m$ nodes and having $k$ rounds \cite[Definition 3.3.2]{HHbook2}, where
\pa\ is offered only one move, namely, {\it a cylindrifier move}: 
Suppose that we are at round $t>0$. Then \pa\ picks a previously played network $N_t$ $(\nodes(N_t)\subseteq m$), 
$i<n,$ $a\in \At\A$, $\bar{x}\in {}^n\nodes(N_t)$, such that $N_t(\bar{x})\leq {\sf c}_ia$. For her response, \pe\ has to deliver a network $M$
such that $\nodes(M)\subseteq m$,  $M\equiv _i N$, and there is $\bar{y}\in {}^n\nodes(M)$
that satisfies $\bar{y}\equiv _i \bar{x}$ and $M(\bar{y})=a$.  

We write $G_k(\At\A)$, or simply $G_k$, for $G_k^m(\At\A)$ if $m\geq \omega$.
\item The $\omega$--rounded game $\bold G^m(\At\A)$ or simply $\bold G^m$ is like the game $G_{\omega}^m(\At\A)$ 
except that \pa\ has the option 
to reuse the $m$ nodes in play.
\end{enumarab}
\end{definition}
{\it Observe that for $k,m\leq \omega$, the games  $G_k^m(\At\A)$ and $\bold G^m(\At\A)$ depend on the signature of $\A$.}
\begin{definition}\label{sub2} 
Fix $2<n<m$. Assume that $\C\in\CA_m$, $\A\subseteq\mathfrak{Nr}_n\C$ is an
atomic $\CA_n$ and $N$ is an $\A$--network with $\nodes(N)\subseteq m$. Define
$N^+\in\C$ by (with notation as introducted in Definition \ref{sub}):
\[N^+ =
 \prod_{i_0,\ldots, i_{n-1}\in\nodes(N)}{\sf s}_{i_0, \ldots, i_{n-1}}{}N(i_0,\ldots, i_{n-1}).\]
For a network $N$ and  function $\theta$,  the network
$N\theta$ is the complete labelled graph with nodes
$\theta^{-1}(\nodes(N))=\set{x\in\dom(\theta):\theta(x)\in\nodes(N)}$,
and labelling defined by
$$(N\theta)(i_0,\ldots, i_{n-1}) = N(\theta(i_0), \theta(i_1), \ldots,  \theta(i_{n-1})),$$
for $i_0, \ldots, i_{n-1}\in\theta^{-1}(\nodes(N))$.
\end{definition}

For a class $\bold K$ of $\sf BAO$s, we denote by $\bold K^{\sf ad}$ the class of completely additive algebras in $\bold K$.
\begin{lemma}\label{n}
Let $2<n<\omega$, and assume that $m>n$. Let $\sf K$ be any variety between $\Sc_n$ and $\QEA_n$. If  $\A\in \bold S_c\Nr_n\K_m^{\sf ad}$ is atomic, then \pe\ has a \ws\ in $\bold G^m(\At\A).$ 
If $\A\in \sf K$, and $\A$ has a complete $m$-square representation then  \pe\ has a \ws\ in $G^m_{\omega}(\At\A).$  
\end{lemma}
\begin{proof} We give the proof for $\CA$s italicizing the part where additivity is used. The stipulated additivity condition when considering only $\CA$s is superflouos since it holds anyway. 
The proof  lifts  ideas
in \cite[Lemmata 29, 26, 27]{r} formulated for relation algebras to $\CA_n$.
Fix $2<n<m$. Assume that $\C\in\CA_m$, $\A\subseteq_c\mathfrak{Nr}_n\C$ is an
atomic $\CA_n$. 
Then the following hold:

(1): for all $x\in\C\setminus\set0$ and all $i_0, \ldots, i_{n-1} < m$, there is $a\in\At\A$, such that
${\sf s}_{i_0,\ldots, i_{n-1}}a\;.\; x\neq 0$,

(2): for any $x\in\C\setminus\set0$ and any
finite set $I\subseteq m$, there is a network $N$ such that
$\nodes(N)=I$ and $x\cdot N^+\neq 0$,, with notation as inDefinition \ref{sub2}. Furthermore, for any networks $M, N$ if
$M^+\cdot N^+\neq 0$, then
$M\restr {\nodes(M)\cap\nodes(N)}=N\restr {\nodes(M)\cap\nodes(N)},$

(3): if $\theta$ is any partial, finite map $m\to m$
and if $\nodes(N)$ is a proper subset of $m$,
then $N^+\neq 0\rightarrow {(N\theta)^+}\neq 0$. If $i\not\in\nodes(N),$ then ${\sf c}_iN^+=N^+$.


Since $\A\subseteq _c\mathfrak{Nr}_n \C$, then $\sum^{\C}\At\A=1$. 
{\it For (1), ${\sf s}^i_j$ is a
completely additive operator (any $i, j<m$), hence ${\sf s}_{i_0,\ldots, i_{n-1}}$
is, too.}
So $\sum^{\C}\set{{\sf s}_{i_0\ldots, i_{n-1}}a:a\in\At(\A)}={\sf s}_{i_0\ldots i_{n-1}}
\sum^{\C}\At\A={\sf s}_{i_0\ldots, i_{n-1}}1=1$ for any $i_0,\ldots, i_{n-1}<m$.  Let $x\in\C\setminus\set0$.  Assume for contradiction
that  ${\sf s}_{i_0\ldots, i_{n-1}}a\cdot x=0$ for all $a\in\At\A$. Then  $1-x$ will be
an upper bound for $\set{{\sf s}_{i_0\ldots i_{n-1}}a: a\in\At\A}.$
But this is impossible
because $\sum^{\C}\set{{\sf s}_{i_0\ldots, i_{n-1}}a :a\in\At\A}=1.$

To prove the first part of (2), we repeatedly use (1).
We define the edge labelling of $N$ one edge
at a time. Initially, no hyperedges are labelled.  Suppose
$E\subseteq\nodes(N)\times\nodes(N)\ldots  \times\nodes(N)$ is the set of labelled hyperedges of $
N$ (initially $E=\emptyset$) and
$x\;.\;\prod_{\bar c \in E}{\sf s}_{\bar c}N(\bar c)\neq 0$.  Pick $\bar d$ such that $\bar d\not\in E$.
Then by (1) there is $a\in\At(\A)$ such that
$x\;.\;\prod_{\bar c\in E}{\sf s}_{\bar c}N(\bar c)\;.\;{\sf s}_{\bar d}a\neq 0$.
Include the hyperedge $\bar d$ in $E$.  We keep on doing this until eventually  all hyperedges will be
labelled, so we obtain a completely labelled graph $N$ with $N^+\neq 0$.
it is easily checked that $N$ is a network.
For the second part of $(2)$, we proceed contrapositively. Assume that there is
$\bar c \in{}\nodes(M)\cap\nodes(N)$ such that $M(\bar c )\neq N(\bar c)$.
Since edges are labelled by atoms, we have $M(\bar c)\cdot N(\bar c)=0,$
so
$0={\sf s}_{\bar c}0={\sf s}_{\bar c}M(\bar c)\;.\; {\sf s}_{\bar c}N(\bar c)\geq M^+\cdot N^+$.
A piece of notation. For $i<m$, let $Id_{-i}$ be the partial map $\{(k,k): k\in m\smallsetminus\{i\}\}.$
For the first part of (3)
(cf. \cite[Lemma~13.29]{HHbook} using the notation in {\it op.cit}), since there is
$k\in m\setminus\nodes(N)$, \/ $\theta$ can be
expressed as a product $\sigma_0\sigma_1\ldots\sigma_t$ of maps such
that, for $s\leq t$, we have either $\sigma_s=Id_{-i}$ for some $i<m$
or $\sigma_s=[i/j]$ for some $i, j<m$ and where
$i\not\in\nodes(N\sigma_0\ldots\sigma_{s-1})$.
But clearly  $(N Id_{-j})^+\geq N^+$ and if $i\not\in\nodes(N)$ and $j\in\nodes(N)$, then
$N^+\neq 0 \rightarrow {(N[i/j])}^+\neq 0$.
The required now follows.  The last part is straightforward.

Using the above proven facts,  we are now ready to show that \pe\  has a \ws\ in $\bold G^m$. She can always
play a network $N$ with $\nodes(N)\subseteq m,$ such that
$N^+\neq 0$.\\
In the initial round, let \pa\ play $a\in \At\A$.
\pe\ plays a network $N$ with $N(0, \ldots, n-1)=a$. Then $N^+=a\neq 0$.
Recall that here \pa\ is offered only one (cylindrifier) move.
At a later stage, suppose \pa\ plays the cylindrifier move, which we denote by
$(N, \langle f_0, \ldots, f_{n-2}\rangle, k, b, l).$
He picks a previously played network $N$,  $f_i\in \nodes(N), \;l<n,  k\notin \{f_i: i<n-2\}$,
such that $b\leq {\sf c}_l N(f_0,\ldots,  f_{i-1}, x, f_{i+1}, \ldots, f_{n-2})$ and $N^+\neq 0$.
Let $\bar a=\langle f_0\ldots f_{i-1}, k, f_{i+1}, \ldots f_{n-2}\rangle.$
Then by  second part of  (3)  we have that ${\sf c}_lN^+\cdot {\sf s}_{\bar a}b\neq 0$
and so  by first part of (2), there is a network  $M$ such that
$M^+\cdot{\sf c}_{l}N^+\cdot {\sf s}_{\bar a}b\neq 0$.
Hence $M(f_0,\dots, f_{i-1}, k, f_{i-2}, \ldots$ $, f_{n-2})=b$,
$\nodes(M)=\nodes(N)\cup\set k$, and $M^+\neq 0$, so this property is maintained.

Assume that $\A$ is an atomic $\CA_n$ having a complete $m$--square representation.
We will show that \pe\ has a \ws\  in $G_{\omega}^m(\At\A)$.
Let $\Mo$ be a complete $m$--square representation of
 $\A$.
One constructs the $m$--dimensional atomic dilation $\D$ using $L_{\infty, \omega}^n$ formulas
from the  complete $m$--square representation as the algebra with univese ${\sf C}^{m}(M)$ and operations induced by clique guarded semantics.
For each $\bar{a}\in 1^{\D},$ define \cite[Definition 13.22] {HHbook} a labelled
hypergraph $N_{\bar{a}}$ with nodes $m$, and
$N_{\bar{a}}(\bar{x})$ when $|\bar{x}|=n$, is the unique atom of $\A$
containing the tuple of length $m>n$,
$(a_{x_0},\ldots, a_{x_{1}},\ldots, a_{x_{n-1}}, a_{x_0}\ldots,\ldots a_{x_0}).$
It is clear that if $s\in 1^{\D}$ and $i, j<m$,
then $s\circ [i|j]\in 1^{\D}$.
By \cite[Lemma 13.24]{HHbook}  $N_{\bar{a}}$ is a network.
Let $H$ be the symmetric closure
of $\{N_a: \bar{a}\in 1^M\}$, that is $\{N\theta: \theta:m\to m, N\in H\}$.
Then $H$ is an $m$--dimensional basis. 
Now \pe\ can win $G_{\omega}^m$ by always
playing a subnetwork of a network in the constructed $H$.
In round $0$, when \pa\ plays
the atom $a\in \A$, \pe\ chooses $N\in H$ with $N(0,1,\ldots, n-1)=a$ and plays $N\upharpoonright n$.
In round $t>0$, inductively if the current network is $N_{t-1}\subseteq M\in H$, then no matter how \pa\ defines $N$, we have
$N\subseteq M$ and $|N|<m$, so there is $z<m$, with $z\notin \nodes(N)$.
Assume that  \pa\ picks $x_0,\ldots, x_{n-1}\in \nodes(N)$, $a\in \At\A$ and $i<n$ such that
$N(x_0,\ldots, x_{n-1})\leq {\sf c}_ia$, so $M(x_0, \ldots  x_{n-1})\leq {\sf c}_ia$,
and hence (by the properties of $H$), there is $M'\in H$ with
$M'\equiv _i M$ and $M'(x_0, \ldots, z, \ldots,  x_{n-1})=a$, with $z$ in the $i$th place.
Now \pe\ responds with the restriction of $M'$
to $\nodes(N)\cup \{z\}$.

\end{proof}
 In the next Theorem ${\sf LCA}_n$ denotes the class of atomic $\CA_n$s whose atom structures satisfy the Lyndon condition as defined in \cite{HHbook2}. 
It is known that $\sf LCA_n$ is n elementary class admitting no finite first order axiomatization; furthermore ${\sf LCA}_n={\bf El}\CRCA_n$.

\begin{theorem}\label{bsl} Let $\kappa$ be an infinite cardinal. Then there exists a $\C\in \QEA_{\omega}$ such that  for all 
$2<n<\omega$, $|\mathfrak{Nr}_n\C|=2^{\kappa}$, $\mathfrak{Nr}_n\C\in {\sf LQEA}_n$, 
but $\Rd_{df}\mathfrak{Nr}_n\C$ is not completely representable.  
cannot be omitted.  
\end{theorem}
\begin{proof}
One uses the ideas in \cite{bsl} replacing 
$\omega$ and $\omega_1$ by $\kappa$ and $2^{\kappa}$, respectively, constructing $\C$ from a relation algebra. 
The resulting (new) relation algebra $\R$ has an $\omega$ 
dimensional amalgamation class $S$, cf. \cite[Lemma 3]{bsl}.
Using the notation in \cite[Lemma 6]{bsl}, let $\C$ be the subalgebra of $\Ca(S)$ generated by $X'$; the latter is defined just before the lemma.
Then  $\R = \mathfrak{Ra}(\C)$, cf. \cite[Lemmata 6, 7]{bsl}, but $\R$ has no complete representation \cite[Lemma 2]{bsl}.
Then $\mathfrak{Nr}_n\C$ ($2<n<\omega$) is atomic, but has no complete representation. By Lemma \ref{n}, \pe\ has a \ws\ in $\bold G_{\omega}(\At\mathfrak{Nr}_n\C)$, hence she has a \ws\ in   
$G_{\omega}(\At\mathfrak{Nr}_n\C)$, {\it  a fortiori} in $G_k(\At\mathfrak{Nr}_n\C)$
for all $k\in \omega$, hence by coding the \ws's of the  $G_k$'s 
in first order sentences, we get that $\mathfrak{Nr}_n\C$ satisfies these first order sentences which are precisely (by definition) 
the Lyndon conditions.   
We use the following uncountable version of Ramsey's theorem due to
Erdos and Rado:
If $r\geq 2$ is finite, $k$  an infinite cardinal, then
$exp_r(k)^+\to (k^+)_k^{r+1}$
where $exp_0(k)=k$ and inductively $exp_{r+1}(k)=2^{exp_r(k)}$.
The above partition symbol describes the following statement. If $f$ is a coloring of the $r+1$
element subsets of a set of cardinality $exp_r(k)^+$
in $k$ many colors, then there is a homogeneous set of cardinality $k^+$
(a set, all whose $r+1$ element subsets get the same $f$-value).
Let $\kappa$ be the given cardinal. We use a variation a simplified more basic version of a rainbow construction where only 
the two predominent  colours, namely, the reds and blues are available. 
The algebra $\C$ will be constructed from a relation algebra possesing an $\omega$-dimensional cylindric basis.
To define the relation algebra we specify its atoms and the forbidden triples of atoms. The atoms are $\Id, \; \g_0^i:i<2^{\kappa}$ and $\r_j:1\leq j<
\kappa$, all symmetric.  The forbidden triples of atoms are all
permutations of $({\sf Id}, x, y)$ for $x \neq y$, \/$(\r_j, \r_j, \r_j)$ for
$1\leq j<\kappa$ and $(\g_0^i, \g_0^{i'}, \g_0^{i^*})$ for $i, i',
i^*<2^{\kappa}.$ 
Write $\g_0$ for $\set{\g_0^i:i<2^{\kappa}}$ and $\r_+$ for
$\set{\r_j:1\leq j<\kappa}$. Call this atom
structure $\alpha$.  
Consider the term algebra $\A$ defined to be the subalgebra of the complex algebra of this atom structure generated by the atoms.
We claim that $\A$, as a relation algebra,  has no complete representation, hence any algebra sharing this 
atom structure is not completely representable, too. Indeed, it is easy to show that if $\A$ and $\B$ 
are atomic relation algebras sharing the same atom structure, so that $\At\A=\At\B$, then $\A$ is completely representable $\iff$ $\B$ is completely representable.

Assume for contradiction that $\A$ has a complete representation $\Mo$.  Let $x, y$ be points in the
representation with $\Mo \models \r_1(x, y)$.  For each $i< 2^{\kappa}$, there is a
point $z_i \in \Mo$ such that $\Mo \models \g_0^i(x, z_i) \wedge \r_1(z_i, y)$.
Let $Z = \set{z_i:i<2^{\kappa}}$.  Within $Z$, each edge is labelled by one of the $\kappa$ atoms in
$\r_+$.  The Erdos-Rado theorem forces the existence of three points
$z^1, z^2, z^3 \in Z$ such that $\Mo \models \r_j(z^1, z^2) \wedge \r_j(z^2, z^3)
\wedge \r_j(z^3, z_1)$, for some single $j<\kappa$.  This contradicts the
definition of composition in $\A$ (since we avoided monochromatic triangles).
Let $S$ be the set of all atomic $\A$-networks $N$ with nodes
$\omega$ such that $\{\r_i: 1\leq i<\kappa: \r_i \text{ is the label
of an edge in $N$}\}$ is finite.
Then it is straightforward to show $S$ is an amalgamation class, that is for all $M, N
\in S$ if $M \equiv_{ij} N$ then there is $L \in S$ with
$M \equiv_i L \equiv_j N$, witness \cite[Definition 12.8]{HHbook} for notation.
Now let $X$ be the set of finite $\A$-networks $N$ with nodes
$\subseteq\kappa$ such that:

\begin{enumerate}
\item each edge of $N$ is either (a) an atom of
$\A$ or (b) a cofinite subset of $\r_+=\set{\r_j:1\leq j<\kappa}$ or (c)
a cofinite subset of $\g_0=\set{\g_0^i:i<2^{\kappa}}$ and

\item  $N$ is `triangle-closed', i.e. for all $l, m, n \in \nodes(N)$ we
have $N(l, n) \leq N(l,m);N(m,n)$.  That means if an edge $(l,m)$ is
labelled by $\sf Id$ then $N(l,n)= N(m,n)$ and if $N(l,m), N(m,n) \leq
\g_0$ then $N(l,n)\cdot \g_0 = 0$ and if $N(l,m)=N(m,n) =
\r_j$ (some $1\leq j<\omega$) then $N(l,n)\cdot \r_j = 0$.
\end{enumerate}
For $N\in X$ let $\widehat{N}\in\Ca(S)$ be defined by
$$\set{L\in S: L(m,n)\leq
N(m,n) \mbox{ for } m,n\in \nodes(N)}.$$
For $i\in \omega$, let $N\restr{-i}$ be the subgraph of $N$ obtained by deleting the node $i$.
Then if $N\in X, \; i<\omega$ then $\widehat{\cyl i N} =
\widehat{N\restr{-i}}$.
The inclusion $\widehat{\cyl i N} \subseteq (\widehat{N\restr{-i})}$ is clear.
Conversely, let $L \in \widehat{(N\restr{-i})}$.  We seek $M \equiv_i L$ with
$M\in \widehat{N}$.  This will prove that $L \in \widehat{\cyl i N}$, as required.
Since $L\in S$ the set $T = \set{\r_i \notin L}$ is infinite.  Let $T$
be the disjoint union of two infinite sets $Y \cup Y'$, say.  To
define the $\omega$-network $M$ we must define the labels of all edges
involving the node $i$ (other labels are given by $M\equiv_i L$).  We
define these labels by enumerating the edges and labeling them one at
a time.  So let $j \neq i < \kappa$.  Suppose $j\in \nodes(N)$.  We
must choose $M(i,j) \leq N(i,j)$.  If $N(i,j)$ is an atom then of
course $M(i,j)=N(i,j)$.  Since $N$ is finite, this defines only
finitely many labels of $M$.  If $N(i,j)$ is a cofinite subset of
$\g_0$ then we let $M(i,j)$ be an arbitrary atom in $N(i,j)$.  And if
$N(i,j)$ is a cofinite subset of $\r_+$ then let $M(i,j)$ be an element
of $N(i,j)\cap Y$ which has not been used as the label of any edge of
$M$ which has already been chosen (possible, since at each stage only
finitely many have been chosen so far).  If $j\notin \nodes(N)$ then we
can let $M(i,j)= \r_k \in Y$ some $1\leq k < \kappa$ such that no edge of $M$
has already been labelled by $\r_k$.  It is not hard to check that each
triangle of $M$ is consistent (we have avoided all monochromatic
triangles) and clearly $M\in \widehat{N}$ and $M\equiv_i L$.  The labeling avoided all
but finitely many elements of $Y'$, so $M\in S$. So
$\widehat{(N\restr{-i})} \subseteq \widehat{\cyl i N}$.

Now let $\widehat{X} = \set{\widehat{N}:N\in X} \subseteq \Ca(S)$.
Then we claim that the subalgebra of $\Ca(S)$ generated by $\widehat{X}$ is simply obtained from
$\widehat{X}$ by closing under finite unions.
Clearly all these finite unions are generated by $\widehat{X}$.  We must show
that the set of finite unions of $\widehat{X}$ is closed under all cylindric
operations.  Closure under unions is given.  For $\widehat{N}\in X$ we have
$-\widehat{N} = \bigcup_{m,n\in \nodes(N)}\widehat{N_{mn}}$ where $N_{mn}$ is a network
with nodes $\set{m,n}$ and labeling $N_{mn}(m,n) = -N(m,n)$. $N_{mn}$
may not belong to $X$ but it is equivalent to a union of at most finitely many
members of $\widehat{X}$.  The diagonal $\diag ij \in\Ca(S)$ is equal to $\widehat{N}$
where $N$ is a network with nodes $\set{i,j}$ and labeling
$N(i,j)=\sf Id$.  Closure under cylindrification is given.
Let $\C$ be the subalgebra of $\Ca(S)$ generated by $\widehat{X}$.
Then $\A = \mathfrak{Ra}(\C)$.
To see why, each element of $\A$ is a union of a finite number of atoms,
possibly a co--finite subset of $\g_0$ and possibly a co--finite subset
of $\r_+$.  Clearly $\A\subseteq\mathfrak{Ra}(\C)$.  Conversely, each element
$z \in \mathfrak{Ra}(\C)$ is a finite union $\bigcup_{N\in F}\widehat{N}$, for some
finite subset $F$ of $X$, satisfying $\cyl i z = z$, for $i > 1$. Let $i_0,
\ldots, i_k$ be an enumeration of all the nodes, other than $0$ and
$1$, that occur as nodes of networks in $F$.  Then, $\cyl
{i_0} \ldots
\cyl {i_k}z = \bigcup_{N\in F} \cyl {i_0} \ldots
\cyl {i_k}\widehat{N} = \bigcup_{N\in F} \widehat{(N\restr{\set{0,1}})} \in \A$.  So $\mathfrak{Ra}(\C)
\subseteq \A$.
$\A$ is relation algebra reduct of $\C\in\CA_\omega$ but has no complete representation.
But in fact $\C$ is in $\QEA_{\omega}$. Let $n>2$. Let $\B=\Nrr_n \C$. Then
$\B\in {\sf Nr}_n\QEA_{\omega}$, is atomic, but even its $\sf Df$ reduct has no complete representation for plainly a complete representation of $\Rd_{df}\B$ induces one of $\B$ hence one for $\A$. 
In fact, because $\B$  is generated by its two dimensional elements,
and its dimension is at least three, its
$\Df$ reduct is not completely representable.
\cite[Proposition 4.10]{Hodkinson}.
It remains to show that the $\omega$--dilation $\C$ is atomless. 
For any $N\in X$, we can add an extra node 
extending
$N$ to $M$ such that $\emptyset\subsetneq M'\subsetneq N'$, so that $N'$ cannot be an atom in $\C$.
\end{proof}
%
\begin{lemma} Let $2<n<\omega$. 
If $\A$ is atomic and $\A\in \Nr_n\QEA_{\omega}$ then $\A\in {\sf LQEA}_n$.
An entirely analogous result holds for relation algebras upon replacing $\Nr_n\CA_{\omega}$ by $\Ra\CA_{\omega}$.
\end{lemma}
\begin{proof}
Assume that $\A$ is as in the hypothesis. Being in the class $\Nr_n\QEA_{\omega}\subseteq (\bold S_c\Nr_n\QEA_{\omega})$. By 
Lemma \ref{n}, \pe\ has a \ws\ in $\bold G^{\omega}\At\A$. Since infinitely many nodes are used (and reuse), 
hence she has a \ws\ in the usual $\omega$ rounded usual atomic $G_{\omega})\At\A)$ without the need to reuse th nodes in play, {\it a fortori} she has a \ws\ in the $k$ rounded atomic game
$G_k(\At\A)$ for all $k\in \omega$. By definition, $\A\in {\sf LQEA}_n$.
\end{proof}

In the previous construction used in Proposition \ref{bsl}, and the previous Lemma $\A\in \sf Ra\CA_{\omega}$ 
and $\B\in \Nr_n\CA_{\omega}$  satisfy the Lyndon conditions, but are not completely representable. Thus:
\begin{corollary}\cite{HH}\label{HH} Let $2<n<\omega$. Then the classes $\sf CRRA$ and for any variety $\sf V$ between ${\sf Df}_n$ and $\QEA_n$ ${\sf CRV}$ is 
not elementary.
\end{corollary}


\begin{thebibliography}{}

\bibitem{1} H. Andr\'eka, M. Ferenczi and  I. N\'emeti, (Editors), {\bf Cylindric-like Algebras and Algebraic Logic},
Bolyai Society Mathematical Studies and Springer-Verlag, {\bf 22} (2012).


\bibitem{ANT}  H. Andr\'eka,  I. N\'emeti and T. Sayed Ahmed,  {\it Omitting types for finite variable fragments and complete representations.}
Journal of Symbolic Logic. {\bf 73} (2008) pp. 65--89.


\bibitem{DM}  A. Daigneault and J.D. Monk,
{\it Representation Theory for Polyadic algebras}, Fundamenta  Mathematica, {\bf 52}(1963), p.151--176.


\bibitem{f} Ferenczi M. {\it A new representation theory: Representing cylindric-like algebras by relativzed set algebras} in \cite{1} p. 135-162

\bibitem{Halmos} Halmos, P., {\it Algebraic Logic.} 
Chelsea Publishing Co., New York, (1962.)



\bibitem{r} R. Hirsch, {\it Relation algebra reducts of cylindric algebras and complete representations},
Journal of Symbolic Logic, {\bf 72}(2) (2007), p.673--703.


\bibitem{HMT2}  L. Henkin, J.D. Monk and  A. Tarski {\it Cylindric Algebras Part I}.
North Holland, 1985.

\bibitem{HH} R. Hirsch and I. Hodkinson {\it Complete representations in algebraic logic},
Journal of Symbolic Logic, {\bf 62}(3)(1997) p. 816--847.

\bibitem{HHbook}  R. Hirsch and I. Hodkinson,  {\it Relation algebras by games.}
Studies in Logic and the Foundations of Mathematics, {\bf 147} (2002).

\bibitem{HHbook2} R. Hirsch and I. Hodkinson {\it  Completions and complete representations}, in \cite{1} pp. 61--90.

\bibitem{HHM}  R. Hirsch, I. Hodkinson, and R. Maddux,
{\it Relation algebra reducts of cylindric algebras
and an application to proof theory,} Journal of Symbolic Logic
{\bf 67}(1) (2002), p. 197--213.


\bibitem{t} R. Hirsch and T. Sayed Ahmed, {\it The neat embedding problem for algebras other than cylindric algebras
and for infinite dimensions.} Journal of Symbolic Logic {\bf 79}(1) (2014), pp .208--222.
\bibitem{Hodkinson} I. Hodkinson, {\it Atom structures of relation and cylindric algebras}. Annals of pure and applied logic,
{\bf 89}(1997), p.117--148.

\bibitem{j} J.S Johnson {\it Nonfinitizability of classes of Polyadic Algebras} Journal of Symbolic Logic {\bf 34}(3) (1969), pp. 344-352.
\bibitem{K} Keisler H.J., {\it A complete first order logic with infinitary predicates} Fund. Math {\bf 52}(1963) p.177-203
\bibitem{Sagi} Sagi, G {\it Polyadic algebras} In \cite{1}.

\bibitem{Pinter} C. Pinter , {\it Cylindric algebras and algebras of substitutions}. Transactions of the American Mathematical Society, 
{\bf 175}(1973), pp. 167--179.

\bibitem{ST} Sain I. and Thompson R.,
{\it Strictly finite schema axiomatization of quasi-polyadic algebras}. In `Algebraic Logic' North Holland, Editors 
Andr\'eka H., Monk D., and N\'emeti I. pp. 539--572. 





\bibitem{IGPL} T. Sayed Ahmed  {\it The class of neat reducts is not elementary.} Logic Journal of $IGPL$, {\bf 9}(2001), pp. 593--628.

\bibitem{Fm}  T. Sayed Ahmed  {\it The class of $2$-dimensional polyadic algebras is not elementary},
Fundamenta Mathematica, 
{\bf 172} (2002), pp. 61--81.

\bibitem{MLQ}  T. Sayed Ahmed,  {\it A model-theoretic solution to a problem of Tarski.} Math Logic Quarterly. {\bf 48}(2002), pp. 343--355.


\bibitem{note} T. Sayed Ahmed, {\it A note on neat reducts.} Studia Logica, {\bf 85} (2007), pp. 139-151.

\bibitem{bsl} T. Sayed Ahmed, {\it $\Ra\CA_n$ is not elementary for $n\geq 5$} 
Bulletin Section of Logic. {\bf 37}(2)(2008) pp. 123--136.



\bibitem{bsl} T. Sayed Ahmed,  {\it Neat embedding is not sufficient for complete representations}
Bulletin Section of Logic {\bf 36}(1) (2007) pp. 29--36.

\bibitem{Sayedneat}  T. Sayed Ahmed, {\it Neat reducts and neat embeddings in cylindric algebras}, in \cite{1}, pp. 105--134.

\bibitem{Sayed}  T. Sayed Ahmed  {\it Completions, Complete representations and Omitting types}, in \cite{1}, pp. 186--205.

\bibitem{au} T. Sayed Ahmed  {\it The class of completely representable polyadic algebras of infinite dimension
is elementary} Algebra Universalis (in press).


\bibitem{mlq} T. Sayed Ahmed, {\it On notions of representability for cylindric--polyadic algebras and a solution to the finitizability problem for first order logic with equality}. 
Mathematical Logic Quarterly, in press.

\bibitem{jsl} T. Sayed Ahmed {\it Non elementary classes in algebraic logic} arxiv


\bibitem{SL}  T. Sayed Ahmed  and I. N\'emeti,  {\it On neat reducts of algebras of logic}, Studia Logica. {\bf 68(2)} (2001), pp. 229--262.

\end{thebibliography}
\end{document}